\documentclass[reqno]{amsart} 

\usepackage[T1]{fontenc}
\usepackage[utf8]{inputenc}
\usepackage[english]{babel}
\usepackage[sc]{mathpazo} 
\usepackage{eulervm} 

\usepackage{etoolbox} 

\usepackage{mathtools} 
\mathtoolsset{showonlyrefs} 

\usepackage{amssymb,textcomp,mathrsfs,stmaryrd,bm,braket,commath}
\SetSymbolFont{stmry}{bold}{U}{stmry}{m}{n} 

\usepackage[marginpar]{todo}

\usepackage{tikz}
\usetikzlibrary{arrows.meta}
\tikzset{every picture/.style={thick,>=latex}} 

\theoremstyle{plain} 
\newtheorem{thm}{Theorem}[section]
\newtheorem{prop}{Proposition}[section]
\newtheorem{lem}{Lemma}[section]
\newtheorem{cor}{Corollary}[section]

\theoremstyle{definition}
\newtheorem{defn}{Definition}[section]

\theoremstyle{remark}
\newtheorem{rem}{Remark}[section]
\newtheorem*{rem*}{Remark}

\newcommand{\addQEDstyle}[2]{\AtBeginEnvironment{#1}{\pushQED{\qed}\renewcommand{\qedsymbol}{#2}}
\AtEndEnvironment{#1}{\popQED}} 

\addQEDstyle{rem}{$\triangle$}
\addQEDstyle{rem*}{$\triangle$} 

\DeclareMathOperator{\vol}{vol}

\DeclareMathOperator{\Tr}{Tr}

\DeclareMathOperator{\SU}{SU}

\DeclareMathOperator{\SL}{SL}
\DeclareMathOperator{\Lie}{Lie}
\DeclareMathOperator{\dR}{dR}

\DeclareMathOperator{\Ad}{Ad}
\DeclareMathOperator{\fl}{fl}
\DeclareMathOperator{\Hom}{Hom}

\DeclareMathOperator{\rk}{rk}

\DeclareMathOperator{\Dol}{Dol}
\DeclareMathOperator{\B}{B}

\DeclareMathOperator{\End}{End}

\DeclareMathOperator{\HW}{HW}
\DeclareMathOperator{\T}{T}
\DeclareMathOperator{\Ext}{Ext}
\renewcommand{\subset}{\subseteq}
\newcommand{\Dom}{\mathcal{D}}
\newcommand{\DomC}{\Dom^{\mathbb{C}}}
\newcommand{\DomQ}{\Dom_{Q}}

\AtBeginDocument{\def\MR#1{}} 

\usepackage[]{hyperref}
\hypersetup{colorlinks, linkcolor=red, filecolor=green, urlcolor=blue, citecolor=blue}

\begin{document}

\author{J\o{}rgen Ellegaard Andersen, Alessandro Malus\`{a}, Gabriele Rembado}

\title[Genus-one complex quantum Chern--Simons]{Genus-one complex quantum Chern--Simons theory} 

\address[J.E. Andersen]{Centre for Quantum Geometry (QM), Danish Institute for Advanced Study, SDU, Denmark}
\email{jea@sdu.dk} 
    
\address[A. Malus\`{a}]{Department of Mathematics, University of Toronto, 40 St. George St., Toronto, ON, M5S 2E4, Canada}
\email{amalusa@math.utoronto.ca} 

\address[G. Rembado]{Hausdorff Centre for Mathematics (HCM), Endenicher Allee 62, D-53115, Bonn, Germany}
\email{gabriele.rembado@hcm.uni-bonn.de}


\begin{abstract}
We consider the geometric quantisation of Chern--Simons theory for closed genus-one surfaces and semisimple complex algebraic groups.
First we introduce the natural complexified analogue of the Hitchin connection in K\"{a}hler quantisation, with polarisations coming from the nonabelian Hodge hyper-K\"{a}hler geometry of the moduli spaces of flat connections, thereby complementing the real-polarised approach of Witten. 
Then we consider the connection of Witten, and we identify it with the complexified Hitchin connection using a version of the Bargmann transform on polarised sections over the moduli spaces.
\end{abstract}

{\let\newpage\relax\maketitle} 
    
\setcounter{tocdepth}{1} 
\tableofcontents

\section*{Introduction and main results}

\renewcommand{\thethm}{\arabic{thm}} 
\renewcommand{\therem}{\arabic{rem}} 

One of the mathematical approaches to the quantisation of classical Chern--Simons theory~\cite{freed_1995_classical_chern_simons_theory_i,freed_2002_classical_chern_simons_theory_ii} is the geometric quantisation of the moduli spaces of flat connections on a surface, as Witten also proposed after introducing it as a quantum field theory~\cite{witten_1989_quantum_field_theory_and_the_jones_polynomial}.
This has first been carried out for the structure group $\SU(n)$~\cite{hitchin_1990_flat_connections_and_geometric_quantisation,axelrod_dellapietra_witten_1991_geometric_quantisation_of_chern_simons_gauge_theory, A1,andersen_2012_hitchin_connection_toeplitz_operators_and_symmetry_invariant_deformation_quantisation}, and for $\SL(n,\mathbb{C})$~\cite{witten_1991_quantization_of_chern_simons_gauge_theory_with_complex_gauge_group}. 
In both cases quantisation relies on a Riemann surface structure on the base, resulting in families of quantum Hilbert spaces parametrised by Teichm\"{u}ller space and identified, up to projective factors, via the holonomy of projectively flat connections: the Hitchin connection for $\SU(n)$, and the Hitchin--Witten connection for $\SL(n,\mathbb{C})$, later reformulated in a more general setting~\cite{andersen_2012_hitchin_connection_toeplitz_operators_and_symmetry_invariant_deformation_quantisation,andersen_gammelgaard_2014_the_hitchin_witten_connection_and_complex_quantum_chern_simons_theory, AR}.
The former relies on K\"{a}hler polarisation, while the latter on real ones; a family of K\"{a}hler polarisations also exists for $\SL(n,\mathbb{C})$, and comes from the \emph{hyper-K\"{a}hler} nonabelian Hodge structure of the moduli space. Hence one may try to define a complexified analogue of the Hitchin connection in the complexified setting, which stands as an open problem for arbitrary genus. 

\vspace{5pt}

In this paper we unify these two approaches to the geometric quantisation of complex Chern--Simons theory for genus-one closed surfaces, considering the natural complexified analogue of the Hitchin connection and relating it to the Hitchin--Witten connection via the Bargmann transform.

\vspace{5pt}

Let us briefly recall some aspects of the general high-genus theory. For integers $n,g \geq 2$ denote $K = \SU(n)$, $K^{\mathbb{C}} = \SL(n,\mathbb{C})$, and let $\Sigma$ be a smooth genus-$g$ closed oriented surface. 
The symplectic moduli space $\mathcal{M}_{\fl}$ of isomorphism classes of irreducible flat $K$-connections on $\Sigma$ has natural prequantum data~\cite{freed_1995_classical_chern_simons_theory_i,freed_2002_classical_chern_simons_theory_ii}, and if $\Sigma$ is endowed with a Riemann surface structure then $\mathcal{M}_{\fl}$ inherits a K\"{a}hler structure from the identification with the moduli space of isomorphism classes of topologically trivial stable holomorphic $K^{\mathbb{C}}$-bundles on $\Sigma$~\cite{narasimhan_seshadri_1965_stable_and_unitary_vector_bundles_on_a_compact_riemann_surface,donaldson_1983_a_new_proof_of_a_theorem_of_narasimhan_and_seshadri}. 
One can then apply K\"{a}hler quantisation for every level $k \in \mathbb{Z}_{>0}$, resulting in the space of nonabelian theta functions.

As the Riemann structure on $\Sigma$ is deformed, these spaces fit into a vector bundle on which the Hitchin connection is defined~\cite{hitchin_1990_flat_connections_and_geometric_quantisation,axelrod_dellapietra_witten_1991_geometric_quantisation_of_chern_simons_gauge_theory,andersen_2012_hitchin_connection_toeplitz_operators_and_symmetry_invariant_deformation_quantisation}.

\vspace{5pt}

Starting with the complex group $K^{\mathbb{C}}$ one has instead the holomorphic symplectic de Rham space $\mathcal{M}_{\dR}$, i.e. the moduli space of isomorphism classes of irreducible flat $K^{\mathbb{C}}$-connections on $\Sigma$.
The Chern--Simons action functional now depends on a \emph{complex} quantum level $t = k + is$, and it is used to yield prequantum data for a \emph{real} symplectic structure $\omega_{t}$ on the moduli space. 
A real polarisation can then be introduced using a Riemann structure on the base, and the canonicity of this construction is achieved by the Hitchin--Witten connection~\cite{witten_1991_quantization_of_chern_simons_gauge_theory_with_complex_gauge_group,andersen_gammelgaard_2014_the_hitchin_witten_connection_and_complex_quantum_chern_simons_theory}.

On the other hand K\"{a}hler polarisations can be obtained from nonabelian Hodge theory~\cite{hitchin_1987_the_self_duality_equations_on_a_riemann_surface,donaldson_1987_twisted_harmonic_maps_and_the_self_duality_equations,corlette_1988_flat_g_bundles_with_canonical_metrics,simpson_1992_higgs_bundles_and_local_systems}. 
If $\Sigma$ is a Riemann surface then there is a non-biholomorphic diffeomorphism $\mathcal{M}_{\dR} \simeq \mathcal{M}_{\Dol}$ with the Dolbeault space of isomorphism classes of topologically trivial stable $K^{\mathbb{C}}$-Higgs bundles on $\Sigma$, the nonabelian Hodge correspondence, endowing $\mathcal{M}_{\dR}$ with a second complex structure which can be completed to a hyper-K\"{a}hler triple. 
This is a hyper-K\"{a}hler rotation within the moduli space of isomorphism classes of stable solutions of the self-duality/Hitchin equations on $\Sigma$~\cite{hitchin_1987_the_self_duality_equations_on_a_riemann_surface}, and in this viewpoint $\omega_{t} \in \Omega^{2}(\mathcal{M}_{\dR},\mathbb{R})$ becomes a K\"{a}hler form for a complex structure extracted from the hyper-K\"{a}hler sphere.
Moreover taking monodromy data yields a non-algebraic biholomorphism $\mathcal{M}_{\dR} \simeq \mathcal{M}_{\B}$ with the Betti space, i.e. with the $K^{\mathbb{C}}$-character variety of $\Sigma$; this completes the triple of nonabelian cohomological theories on $\Sigma$~\cite{simpson_1994_moduli_of_representations_of_the_fundamental_group_of_a_smooth_projective_variety_i,simpson_1994_moduli_of_representations_of_the_fundamental_group_of_a_smooth_projective_variety_ii}, and yields the most useful description of the moduli space for our purpose.

\vspace{5pt}

Hereafter we consider the exceptional case $g = 1$, where the above needs to be modified to account for the emptyness of the irreducible locus. Nonetheless we consider natural finite-dimensional descriptions of the moduli spaces, as orbifold quotients of the subspaces of translation-invariant 1-forms with values in the Lie subalgebras of diagonal matrices.

Importantly the same description applies to more general algebraic/Lie groups than $\SU(n) \subseteq \SL(n,\mathbb{C})$, using maximal toral/Cartan subalgebras, with a caveat: for a connected simply-connected semisimple complex algebraic group $K^{\mathbb{C}}$ we will in general obtain a description of the \emph{normalisation} of the De Rham and Betti spaces---as complex algebraic varieties~\cite{thaddeus_2001_mirror_symmetry_langlands_duality_and_commuting_elements_of_lie_groups}---a viewpoint which lends itself to geometric quantisation. 
Note the normalisation map is an isomorphism for all classical groups~\cite{sikora_2014_character_varieties_of_abelian_groups}.

\vspace{5pt}

The first goal is to introduce (projectively) flat connections to relate the quantum Hilbert spaces arising from the aforementioned K\"{a}hler polarisations, passing through the finite-dimensional covering space of $\mathcal{M}_{\dR}$.
In the case of a finite-dimensional affine symplectic space there are two natural constructions~\cite{axelrod_dellapietra_witten_1991_geometric_quantisation_of_chern_simons_gauge_theory,woodhouse_1980_geometric_quantisation}.
The first, a close analogue of
the Hitchin connection, is defined as a covariant differential whose potential is essentially the variation of the Laplace--Beltrami operator; the second uses the orthogonal projection on the $\operatorname{L}^{2}$-closed subspace of holomorphic functions (the Segal--Bargmann space~\cite{bargmann_1961_hilbert_space_of_analytic_functions_and_associated_integral_transform,segal_1963_mathematical_problems_of_relativistic_physics}).

We thus consider these two connections on the
covering space, and restrict them to the families of linear K\"{a}hler polarisations arising from Teichm\"{u}ller space.
We call them the \emph{lifted} complexified Hitchin connection and the $\operatorname{L}^{2}$-connection, respectively. 
Importantly however we derive the former from a connection intrinsically defined for sections over $\mathcal{M}_{\dR}$, generalising~\cite[Chap~4]{rembado_2018_quantisation_of_moduli_spaces_and_connections}, and we show the following by establishing algebraic relations between the differential operators.

\begin{thm}[\S~\ref{sec:complexified_hitchin_connection}]
\label{thm:complexified_hitchin_connection_intro}
    The complexified Hitchin connection preserves holomorphicity.
    Moreover it is flat and mapping class group invariant.
\end{thm}

Next we investigate the relations between this K\"{a}hler quantisation scheme and Witten's approach with real polarisations, i.e. with the genus-one Hitchin--Witten connection.
Over the covering space, the quantum Hilbert spaces of geometric quantisation are isometric via the Bargmann transform~\cite{guillemin_sternberg_1977_geometric_asymptotics,woodhouse_1980_geometric_quantisation}.
In our case, having a family of polarisations of each kind, we introduce a family of Bargmann transforms.
We lift the Hitchin-Witten connection to act on families of sections of the line bundle over the covering space and establish the following---passing through the $\operatorname{L}^{2}$-connection.

\begin{thm}[\S~\ref{sec:conjugation_hitchin_witten}]
\label{thm:l^2=HW}
    The Bargmann transform intertwines the lifted Hitchin--Witten and complexified Hitchin connections when acting on (sufficiently regular) families of polarised sections, in a mapping class group equivariant fashion.
\end{thm}

Next, we focus on relating the Hilbert spaces obtained from geometric quantisation on the moduli spaces themselves.
Adapting the ideas in~\cite{andersen_1992_jones_witten_theory}, which relates the quantization of compact tori with respect to arbitrary linear polarizations, we notice that the Bargmann transform extends to lifts of smooth sections on $\mathcal{M}_{\fl}$ and obtain the following result.

\begin{thm}[\S~\ref{sec:Barg_below}]
	\label{thm:Barg_below}
	For every value of the Teichmüller parameter, the Bargmann transform defines a unitary isomorphism between the corresponding Hilbert spaces arising from geometric quantisation on the moduli spaces.
\end{thm}

Finally, we relate the quantum connections as intrinsically defined on the moduli spaces.
We consider the Bargmann-stable subspaces of Schwartz-class sections of the prequantum line bundles (complex-analytic in the complexified case), and the induced transpose Bargmann transform between their topological duals, and dual versions of the Hitchin--Witten and complexified Hitchin connections.
We then embed $\operatorname{L}^{2}$ polarised sections over the moduli spaces as tempered distributions and show compatibility of all these objects, finally obtaining the following.

\begin{thm}[\S~\ref{sec:duals}]
	\label{thm:CH=HW}
	The Bargmann transform on the moduli spaces intertwines the Hitchin-Witten and complexified Hitchin connections.
\end{thm}

\begin{rem}
	It should be remarked that the potentials of the Hitchin-Witten and complexified Hitchin connections are skew-adjoint, so the Bargmann transform can be viewed as an equivalence of \emph{unitary} connections.
	This is yet again an exceptional feature of the genus $1$, as these connections are typically not compatible with the na\"ive Hermitian structure in general.
\end{rem}

Note the computation of the genus one complex quantum Chern--Simons mapping class group representation of~\cite{AM}, which was based on Witten's explicit description of the covariant constant sections of the Hitchin--Witten connection~\cite{witten_1991_quantization_of_chern_simons_gauge_theory_with_complex_gauge_group}, now also applies to the complexified Hitchin connection by Thm.~\ref{thm:CH=HW} above.

\vspace{5pt}

Finally we would like to add that there are related, but different works in a similar context in \cite{BMN,FMN1,FMN2,FMMN, KMN, Welters}.

\section*{Acknowledgements}

All authors thank the anonymous referee for their observations and suggestions, which led to an overall better version of the present work.

They would also like to thank the former Centre for Quantum Geometry of Moduli Spaces (QGM) at the Aarhus University, and the new Centre for Quantum Mathematics (QM) at the University of Southern Denmark for hospitality and support.

The first-named author was supported in part by the Danish National Science Foundation Center of Excellence grant, Centre for Quantum Geometry of Moduli spaces, DNRF95 and by  the ERC-SyG project, Recursive and Exact New Quantum Theory (ReNewQuantum) which receives funding from the European Research Council (ERC) under the European Union's Horizon 2020 research and innovation programme under grant agreement No 810573.

The second-named author thanks for their support the University of Toronto, the University of Saskatchewan, the Pacific Institute for the Mathematical Sciences (PIMS), the Centre for Quantum Topology and its Applications (quanTA), and NSERC.
He also wishes to thank Steven Rayan for many discussions.

The third-named author was supported by the grant number 178794 of the Swiss National Science Foundation (SNSF), and by the National Centre of Compentence in Research SwissMAP, of the SNSF. 

\renewcommand{\thethm}{\arabic{section}.\arabic{thm}} 
\renewcommand{\therem}{\arabic{section}.\arabic{rem}} 

\section{Moduli spaces}
\label{sec:moduli_spaces}

Let $\Sigma$ be a smooth oriented closed surface of genus $1$, $K$ a compact connected simply-connected Lie group, $K \hookrightarrow K^{\mathbb{C}}$ a complexification, and $\mathfrak{k}^{(\mathbb{C})} \coloneqq \Lie \bigl(K^{(\mathbb{C})}\bigr)$ their Lie algebras,%
\footnote{Hereafter a superscript "$(\mathbb{C})$" denotes presence/absence of a superscript "$\mathbb{C}$".}
and set $r = \rk(K)$.

\subsection{Flat connections}

Let $\mathcal{A}^{(\mathbb{C})}$ be the space of connections on the trivial principal $K^{(\mathbb{C})}$-bundle $P^{(\mathbb{C})} = \Sigma \times K^{(\mathbb{C})} \to \Sigma$, with the Atiyah--Bott symplectic form~\cite{atiyah_bott_1983_yang_mills_equations_over_riemann_surfaces}
\begin{equation}
\label{eq:atiyah_bott_form}
    \widetilde{\omega}^{(\mathbb{C})}(A,B) \coloneqq \int_{\Sigma} \langle A \wedge B \rangle_{\mathfrak{k}^{(\mathbb{C})}} \, , \qquad \text{for } A,B \in \Omega^{1}\bigl(\Sigma,\mathfrak{k}^{(\mathbb{C})}\bigr) \, . 
\end{equation} 
Here $\langle \cdot \wedge \cdot \rangle_{\mathfrak{k}^{(\mathbb{C})}}$ is the contraction with a suitable multiple of the Cartan--Killing form of $\mathfrak{k}^{(\mathbb{C})}$.
The pairing on $\mathfrak{k}$ is such that the cohomology class of the canonical $3$-form lies in $H^{3}_{\dR}(K,2\pi \mathbb{Z})$, and that on $\mathfrak{k}^{\mathbb{C}}$ is obtained by complexification.

The moduli space $\mathcal{M}^{(\mathbb{C})}_{\fl}$ of isomorphism classes of flat $K^{(\mathbb{C})}$-connections is then the level-zero symplectic reduction of $\mathcal{A}^{(\mathbb{C})}$ with respect to the Hamiltonian action of the gauge group $\mathcal{K}^{(\mathbb{C})}$, identifying the moment map with the curvature.

We denote by $\omega^{(\mathbb{C})}$ the form on the reduction, a stratified symplectic space with singular points corresponding to degenerate gauge orbits. Below we will describe it explicitly; see~\cite{franco_garciaprada_newstead_2014_higgs_bundles_over_elliptic_curves,franco_garciaprada_newstead_2019_higgs_bundles_over_elliptic_curves_for_complex_reductive_groups} for the viewpoint of Higgs bundles.

\begin{rem}[De Rham structure]
\label{rem:reduced_tangent_spaces}

    The space $\mathcal{A}^{\mathbb{C}}$ comes with a linear complex structure $\widetilde{J}$ for which~\eqref{eq:atiyah_bott_form} is of type $(2,0)$: 
    \begin{equation}
    \label{eq:linear_complex_structure}
        \widetilde{J}(\alpha \otimes X) \coloneqq \alpha \otimes (i X) \, , \qquad \text{for } \alpha \in \Omega^1(\Sigma,\mathbb{R}), \, X \in \mathfrak{k}^{\mathbb{C}} \, .
    \end{equation} 

    A model for the tangent space at a smooth point is
    \begin{equation}
    \label{eq:reduced_tangent_spaces}
        \T_{[A]} \mathcal{M}_{\fl}^{(\mathbb{C})} = H^1_{\dR,A}\bigl(\Sigma,\mathfrak{k}^{(\mathbb{C})}\bigr) \coloneqq H^1\bigl(\Omega^{\bullet}\bigl(\Sigma,\Ad P^{(\mathbb{C})}\bigr),d_A\bigr) \, ,
    \end{equation} 
    where $d_A$ is the connection on the adjoint vector bundle induced by $A \in \mathcal{A}^{\mathbb{C}}$.
    
    Since $\widetilde{J}$ commutes with $d_{A}$ it induces a complex structure $J$ on~\eqref{eq:reduced_tangent_spaces}, the \emph{de Rham} structure. 
    We write $\mathcal{M}_{\dR} \coloneqq \bigl(\mathcal{M}_{\fl}^{\mathbb{C}},J\bigr)$ for the resulting de Rham space.
    The notation $\mathcal{M}_{\fl}^{\mathbb{C}}$ refers to the underlying space (possibly equipped with different complex structures, cf. Rem.~\ref{rem:dolbeault}). 
\end{rem}

\subsection{Betti viewpoint}

Flat connections on $P^{(\mathbb{C})} \to \Sigma$ may be described as topologically trivial local $K^{(\mathbb{C})}$-systems on $\Sigma$. 
These are classified by monodromy, so there are identifications with character varieties: 
\begin{equation}
\label{eq:character_varieties}
    \mathcal{M}_{\fl} \simeq \Hom\bigl(\pi_1(\Sigma),K\bigr) \big\slash K \, , \quad \mathcal{M}_{\dR} \simeq \mathcal{M}_{\B} \coloneqq \Hom \bigl(\pi_1(\Sigma),K^{\mathbb{C}}\bigr) \sslash K^{\mathbb{C}} \, ,
\end{equation} 
taking a complex GIT quotient in the latter case (as not all $K^{\mathbb{C}}$-orbits are closed). 

To give an explicit description fix a maximal torus $T \subseteq K$, set $\mathfrak{t} \coloneqq \Lie(T) \subseteq \mathfrak{k}$, and   let $T^{\mathbb{C}} \subseteq K^{\mathbb{C}}$ be the connected subgroup with Lie algebra $\mathfrak{t}^{\mathbb{C}} \coloneqq \mathfrak{t} \otimes \mathbb{C} \subseteq \mathfrak{k}^{\mathbb{C}}$.
Denote $N_K(T) \subseteq K$ the normaliser of $T$ in $K$ and $W \coloneqq N_K(T) \big\slash T$ the Weyl group.    

Now two commuting elements of $K$ sit in a common maximal torus, and after a conjugation are taken inside $T$; the residual action is that of the Weyl group~\cite{borel_1961_sous_groupes_commutatifs_et_torsions_des_groups_de_lie_compacts_connexes}, hence we get a homeomorphism $\mathcal{M}_{\fl} \simeq T \times T \big\slash W$.
Moreover the reduction of~\eqref{eq:atiyah_bott_form} becomes the reduction of the natural translation- and Weyl-invariant symplectic form on the Lie group $T^2$, so the notation will not distinguish the two.

In the case of the complex group $K^{\mathbb{C}}$ we explicitly restrict to representations with values in the prescribed maximal algebraic torus $T^{\mathbb{C}}$, which in particular have closed $K^{\mathbb{C}}$-orbits. 
This yields \emph{completely reducible} representations, whereas there are no \emph{irreducible} ones~\cite{sikora_2014_character_varieties_of_abelian_groups} (see e.g.~\cite{sikora_2012_character_varieties} for definitions). 

\begin{prop}[Normalisation of the Betti space]
    The composition
    \begin{equation}
        T^{\mathbb{C}} \times T^{\mathbb{C}} \simeq \Hom \bigl( \mathbb{Z}^{2},T^{\mathbb{C}} \bigr) \longrightarrow \Hom \bigl( \mathbb{Z}^{2},K^{\mathbb{C}} \bigr) \longrightarrow \mathcal{M}_{\B} \, 
    \end{equation}
    factors through the (set-theoretic) quotient $\widetilde{\mathcal{M}}_{\B} \coloneqq \bigl( T^{\mathbb{C}} \times T^{\mathbb{C}} \bigr) \big\slash W$ for the diagonal action of the Weyl group. 
    The resulting arrow $\chi \colon \widetilde{\mathcal{M}}_{\B} \to \mathcal{M}_{\B}$ is a normalisation map.
\end{prop}

\begin{proof}
    By~\cite{thaddeus_2001_mirror_symmetry_langlands_duality_and_commuting_elements_of_lie_groups,sikora_2014_character_varieties_of_abelian_groups} $\chi$ is a normalisation of the irreducible component of the trivial representation; but since $K^{\mathbb{C}}$ is semisimple and simply-connected $\mathcal{M}_{\B}$ itself is irreducible~\cite{richardson_1988_conjugacy_classes_of_n_tuples_in_lie_algebras_and_algebraic_groups}.
\end{proof}

Further $\widetilde{\mathcal{M}}_{\B}$ carries the reduction of the translation- and Weyl-invariant complex symplectic form on the Lie group $T^{\mathbb{C}} \times T^{\mathbb{C}}$.
This matches up with the genus-one Goldman structure, defined on an open dense subset of $\mathcal{M}_{\B}$, as the normalisation map is symplectic~\cite{sikora_2014_character_varieties_of_abelian_groups}. 
Similarly the de Rham complex structure $J$ on $\mathcal{M}_{\dR} \simeq \mathcal{M}_{\B}$ matches up with the reduction of the invariant complex structure. 

Hereafter we work on the normal singular variety $\widetilde{\mathcal{M}}_{\B} \hookrightarrow \mathcal{M}_{\B}$. 
For the sake of simplicity the notation \emph{will not} distinguish between the moduli spaces and their normalisation, and neither between their complex/symplectic structures, so
\begin{equation}
\label{eq:complex_moduli_space_tori_quotient}
    \mathcal{M}_{\dR} \simeq T^{\mathbb{C}} \times T^{\mathbb{C}} \big\slash W \, .
\end{equation} 

\begin{rem*}
    There are natural embedding/projection arrows $\iota \colon \mathcal{M}_{\fl} \rightleftarrows \mathcal{M}_{\fl}^{\mathbb{C}} \colon \pi$, since the diagonal Weyl group action on $T^{(\mathbb{C})} \times T^{(\mathbb{C})}$ commutes with the factorwise projection $T^{\mathbb{C}} \twoheadrightarrow T$ and inclusion $T \hookrightarrow T^{\mathbb{C}}$.
\end{rem*}

\subsection{Finite-dimensional de Rham viewpoint}
\label{sec:finite_dimensional_presentation}

Letting $\mathcal{A}_{0}^{(\mathbb{C})} \coloneqq H^{1} (\Sigma, \mathfrak{t}^{(\mathbb{C})})$ be the de Rham cohomology with coefficients in $\mathfrak{t}^{(\mathbb{C})}$ as a vector space, there is a natural map $\mathcal{A}_{0}^{(\mathbb{C})} \to \mathcal{M}_{\fl}^{(\mathbb{C})}$, which is (holomorphic) symplectic for the cup product on $\mathcal{A}_{0}^{(\mathbb{C})}$.
The monodromy factors through $\mathcal{A}_{0}^{(\mathbb{C})} \to H^{1} (\Sigma, T^{(\mathbb{C})}) \simeq \Hom(\pi_{1}(\Sigma), T^{(\mathbb{C})})$, and the kernel of both maps is $\mathcal{T}_{0} \coloneqq H^{1} (\Sigma, \Lambda)$, where $\Lambda \coloneqq \ker \left( \exp \colon \mathfrak{t} \to T \right)$.
This gives a description of (the normalisation of) $\mathcal{M}_{\fl}^{(\mathbb{C})}$ as the finite-dimensional quotient of $\mathcal{A}_{0}^{(\mathbb{C})}$ by the discrete group $\mathcal{K}_{0} = \mathcal{T}_{0} \rtimes W$.
In particular, smooth objects on each quotient correspond to $\mathcal{K}_{0}$-equivariant ones on the associated vector space.

\section{Symplectic structure and polarisations}
\label{sec:polarisations}

Being a \emph{holomorphic} symplectic space, $(\mathcal{M}^{\mathbb{C}}_{\fl} , \omega^{\mathbb{C}})$ has no preferred \emph{real} symplectic structure.
However a natural one is obtained after introducing a coupling constant in the Chern-Simons action functional, as discussed by Witten~\cite{witten_1991_quantization_of_chern_simons_gauge_theory_with_complex_gauge_group}.
This complex parameter $t = k + is$ is the level of the theory, with $k > 0$ an integer, and the corresponding real symplectic form is $\omega_{t} \coloneqq \mathbb{R}e (t \omega^{\mathbb{C}})$.

\subsection{K\"{a}hler polarisations}

Let $\mathcal{T} = \mathcal{T}_{\Sigma}$ be the Teichm\"{u}ller space of $\Sigma$, identified with the upper half-plane $\mathbb{H} \subset \mathbb{C}$ as usual.
Each class is represented by a complex structure on $\Sigma$ making it isomorphic to $\mathbb{C} \slash (\mathbb{Z} \oplus \tau \mathbb{Z})$ for some $\tau \in \mathbb{H}$.
Given $\tau$, we denote $X_{\tau}$ the resulting Riemann surface.

\subsubsection{Hyperk\"{a}hler structures}
\label{sec:kaehler_polarisations_i}

The Hodge-$\ast$ operator of $X_{\tau}$ yields a complex structure $\widetilde{I}_{\tau}$ on $\mathcal{A}$, with K\"{a}hler metric $\widetilde{g}_{\tau} = \widetilde{\omega} \cdot \widetilde{I}_{\tau}$ corresponding to the $\operatorname{L}^{2}$-pairing.
Complex structures are then naturally induced on tangent spaces at smooth points of $\mathcal{M}_{\fl}$ by acting on harmonic representatives~\cite{hitchin_1990_flat_connections_and_geometric_quantisation,andersen_gammelgaard_2014_the_hitchin_witten_connection_and_complex_quantum_chern_simons_theory} (cf. Rem.~\ref{rem:reduced_tangent_spaces}).
We will denote $I_{\tau}$ and $g_{\tau}$ the reduced structures on $\mathcal{M}_{\fl}$ and their lifts to $\mathcal{A}_{0}$.

Deforming the conformal structure of $\Sigma$ yields a fibre bundle $\bm{\mathcal{M}}_{\fl} \to \mathcal{T}$ of complex manifolds with fibres $(\mathcal{M}_{\fl},I_{\tau})$. 
It is a fibrewise quotient of the complex vector bundle $\bm{\mathcal{A}}_{0} \to \mathcal{T}$, whose fibres are the complex vector spaces $\bigl(\mathcal{A}_{0},I_{\tau}\bigr)$.

\vspace{5pt}

The construction of the complex structure $I_{\tau}^{\mathbb{C}}$ on $\mathcal{M}_{\fl}^{\mathbb{C}}$ is more subtle, as the Hodge-$\ast$ operator and $\operatorname{L}^{2}$-norm on $\mathcal{A}^{\mathbb{C}}$ are not preserved by the $\mathcal{K}^{\mathbb{C}}$-action.
Instead, it is obtained from the identification with $\mathcal{M}_{\Dol}$ via non-abelian Hodge theory.
For the purpose of this work, it will be enough to mention that $I^{\mathbb{C}}_{\tau}$ is realised on each tangent space $T_{[A]} \mathcal{M}_{\fl} \simeq H^{1}_{\dR,A}(X_{\tau},\mathfrak{k})$ as the action of the Hodge-$\ast$ operator associated to a specific metric adapted to $A$, its harmonic metric, on harmonic representatives.
Note furthermore that, if $A \in \mathcal{A}$ is a flat $K$-connection, then the standard metric is harmonic for $A$, so $I^{\mathbb{C}}_{\tau}$ restricts to $I_{\tau}$ on $T_{[A]} \mathcal{M}_{\fl} \subset T_{[A]} \mathcal{M}_{\fl}^{\mathbb{C}}$.

The complex structure $I^{\mathbb{C}}_{\tau}$ defines a hyperkähler structure together with $J$ and $K_{\tau} \coloneqq I^{\mathbb{C}}_{\tau} \circ J$, and the symplectic forms corresponding to $I^{\mathbb{C}}_{\tau}$ and $K_{\tau}$ are $\mathbb{R}e(\omega^{\mathbb{C}})$ and $- \mathbb{I}m(\omega^{\mathbb{C}})$, respectively.
It follows that $\omega_{t} / \abs{t}$ belongs to the family of symplectic forms defined by the hyperkähler structure, corresponding to
\begin{equation}
	I_{t,\tau} \coloneqq k' I^{\mathbb{C}}_{\tau} + s' K_{\tau}
\end{equation}
with $k' = k/\abs{t}$ and $s' = s/\abs{t}$.
This results in a Kähler manifold $\bigl(\mathcal{M}_{\fl}^{\mathbb{C}},I_{\tau,t},\omega_{t}, g_{\tau}^{\mathbb{C}}\bigr)$.

As in the case of $\mathcal{M}_{\fl}$, as $\tau$ varies one obtains a fibration $\bm{\mathcal{M}}_{\fl}^{\mathbb{C}} \to \mathcal{T}$, which is also a quotient of the (quaternionic) vector bundle $\bm{\mathcal{A}}_{0}^{\mathbb{C}} \to \mathcal{T}$.

\begin{rem}
	\label{rem:LeviCivita}
	As we shall see in \S~\ref{sec:coordinates_and_frames}, in suitable coordinates the metric $g_{\tau}^{(\mathbb{C})}$ is represented by a \emph{constant} tensor, trivialising the Levi-civita connection.
	While the tensor itself depends on $\tau$, the coordinates do not, so the connection is independent of $\tau$.
\end{rem}

\begin{rem}[Dolbeault]
\label{rem:dolbeault}

If $s = 0$ then $I_{\tau,t} = I_{\tau}^{\mathbb{C}}$: this is the \emph{Dolbeault} structure on the moduli space $\mathcal{M}_{\Dol,\tau} = \mathcal{M}_{\Dol}\bigl(X_{\tau},K^{\mathbb{C}}\bigr)$ of isomorphism classes of (polystable) topologically trivial $K^{\mathbb{C}}$-Higgs bundles on $X_{\tau}$.
This case also corresponds to the setup of~\cite{andersen_gukov_pei_2016_the_verlinde_formula_for_higgs_bundles}, which however considers the moduli stack in all genera.
\end{rem}

\begin{rem*}
One may also wish to consider \emph{all} the Kähler structures in the family and study the dependence of the quantisation process below on this choice, e.g. for a fixed $\tau \in \mathcal{T}$.
The same authors have addressed this problem, in the case of a $\operatorname{Sp}(1)$-symmetric hyper-Kähler manifold, in a recent pre-print~\cite{AMR21}.
\end{rem*}

\subsection{Real polarisations}
\label{sec:real_polarisations}

The subspace $\widetilde{P}_{\tau} \coloneqq \Omega^{1,0}\bigl(X_{\tau},\mathfrak{k}^{\mathbb{C}}\bigr) \subseteq \mathcal{A}^{\mathbb{C}}$ is $\widetilde{\omega}_{t}$-Lagrangian, hence it defines a linear real polarisation. 
This descends to an $\omega_{t}$-Lagrangian subspace $P_{\tau, [A]} \coloneqq H^{1,0}_{\Dol,A}\bigl(X_{\tau},\mathfrak{k}^{\mathbb{C}}\bigr) \subseteq T_{[A]} \mathcal{M}^{\mathbb{C}}_{\fl}$ (cf. Rem.~\ref{rem:reduced_tangent_spaces}).

\begin{defn}[Real polarisations]
    We denote $P_{\tau}$ the real polarisation thus induced from the subspace $\widetilde{P}_{\tau} \subseteq \mathcal{A}^{\mathbb{C}}$.
\end{defn}

We will also denote ${P}_{\tau}$ the linear real polarisation on the subspace $\mathcal{A}_{0}^{\mathbb{C}} \subseteq \mathcal{A}^{\mathbb{C}}$, induced from $\widetilde{P}_{\tau}$---whose $\mathcal{K}_{0}$-reduction coincides with the above.

\begin{rem}[Symplectic transverse]
\label{rem:transverse_compact_moduli_space}

    By construction 
    \begin{equation}
        P_{\tau,[A]} \cap \T_{[A]} \mathcal{M}_{\fl} = H^{1,0}_{\Dol,A}\bigl(X_{\tau},\mathfrak{k}^{\mathbb{C}}\bigr) \cap H^{1}_{\dR,A}(X_{\tau},\mathfrak{k}) = \{0\} \, ,
    \end{equation} 
    so the tangent bundle $\T\mathcal{M}_{\fl} \subseteq \T\mathcal{M}^{\mathbb{C}}_{\fl}$ is transverse to the real polarisation $P_{\tau}$, and by dimension count $\T\mathcal{M}^{\mathbb{C}}_{\fl} \simeq \T\mathcal{M}_{\fl} \oplus P_{\tau}$.
    On the covering space, every leaf intersects $\mathcal{A}_{0}$ at exactly one point, and since $\mathcal{K}_{0}$ preserves this subspace the same is true for $\mathcal{M}_{\fl}$.
    Hence the moduli space for the compact group is a global symplectic transverse to the real polarisation $P_{\tau}$.
\end{rem}

\begin{rem}
	We will consider the subspace ${Q}_{\tau} \coloneqq {I}_{\tau} {P}_{\tau} \subseteq \mathcal{A}_{0}^{\mathbb{C}}$, complex conjugate of ${P}_{\tau}$ and $\widetilde{g}^{\mathbb{C}}_{\tau}$-orthogonal to it.
	Under the identification $\mathcal{A}_{0}^{\mathbb{C}} \simeq \mathcal{A}_{0} \otimes \mathbb{C}$ they correspond to $\T_{0,1} \mathcal{A}_{0}$ and $\T_{1,0} \mathcal{A}_{0}$, for the complex structure ${I}_{\tau}$.
	Hence the orthogonal projections of $A \in \mathcal{A}_{0}$ onto $P_{\tau}$ and $Q_{\tau}$ read
	\begin{equation}
		\label{eq:projPQ}
		\pi_{P_{\tau}} (A) = \frac{1}{2} \left( A - K_{\tau} A \right) \, ,
		\qquad \qquad
		\pi_{Q_{\tau}} (A) = \frac{1}{2} \left( A + K_{\tau} A \right) \, .
	\end{equation}
	The above are projective isometries with inverses $\mathbb{1} + K_{\tau}$ and $\mathbb{1} - K_{\tau}$.
\end{rem}

\begin{rem}

    Here the action of the mapping class group $\Gamma = \Gamma_{\Sigma} \simeq \SL(2,\mathbb{Z})$ on Teichm\"{u}ller space amounts to that of the modular group on the upper-half plane $\mathcal{T} \simeq \mathbb{H}$. 

    The polarisations constructed in this section only depend on $\Gamma$-orbits of Teichm\"{u}ller elements, i.e. on the (unmarked) Riemann surface structure. 
    In the K\"{a}hler-polarised setting this also holds for the K\"{a}hler metrics $g^{\mathbb{C}}_{\tau}$ and $g_{\tau}$, as they are obtained via the contraction with a $\mathcal{T}$-independent symplectic form.
\end{rem}

\section{Prequantisation and geometric quantisation}
\label{sec:prequantisation_and_quantisation}

\subsection{Prequantisation}
\label{sec:prequantum_data}

There are natural prequantum data for the real symplectic manifolds $(\mathcal{M}_{\fl},k\omega)$ and $(\mathcal{M}_{\fl}^{\mathbb{C}},\omega_{t})$, compatible with the inclusion $\mathcal{M}_{\fl} \hookrightarrow \mathcal{M}_{\fl}^{\mathbb{C}}$, provided that $k \in \mathbb{Z}_{> 0}$~\cite{freed_1995_classical_chern_simons_theory_i,freed_2002_classical_chern_simons_theory_ii,andersen_gammelgaard_2014_the_hitchin_witten_connection_and_complex_quantum_chern_simons_theory}.
The construction relies on a lift of the gauge group action on the trivial line bundle over the affine space $
\mathcal{A}^{(\mathbb{C})}$, with cocycle defined from the Chern--Simons action functional.
Explicitly, if $\gamma \colon \Sigma \to K^{\mathbb{C}}$ is a gauge transformation, then the lifted action at $A \in \mathcal{A}^{\mathbb{C}}$ is the multiplication by
\begin{equation}
	\label{eq:lifted_K_action}
	\Theta_{\gamma, A, t} \coloneqq \exp \biggl( -\frac{i}{2} \operatorname{Re} \biggl( W_{\Sigma, t} (\gamma) + \int_{\Sigma} \Braket{\Ad_{\gamma^{-1}} A \wedge \theta_{\gamma}}_{\mathfrak{k}^{\mathbb{C}}} \biggr) \biggr)
\end{equation}
where $W_{\Sigma, t}$ is the level-$t$ Wess-Zumino-Witten functional and $\theta_{\gamma}$ is the pull-back of the Maurer-Cartan form via $\gamma$.

We denote $\mathcal{L}_k \to \mathcal{M}_{\fl}$ and $\mathcal{L}^{\mathbb{C}}_{t} \to \mathcal{M}^{\mathbb{C}}_{\fl}$ the resulting line bundles, equipped with Hermitian metrics and with compatible prequantum connections $\nabla_k$ and $\nabla_t$ with curvatures $F_{\nabla_k} = -ik\omega$, $F_{\nabla_t} = -i\omega_{t}$.
In the compact case we find the $k$-fold tensor power of the standard Chern--Simons line bundle, i.e. Quillen's determinant bundle in the viewpoint of $\overline{\partial}$-operators~\cite{quillen_1985_determinantns_of_cauchy_riemann_operators_on_riemann_surfaces}.

Moreover there is an explicit finite-dimensional presentation in terms of the $\mathcal{K}_{0}$-reduction of prequantum data on the covering spaces $\mathcal{A}_{0}^{(\mathbb{C})} \twoheadrightarrow \mathcal{M}^{(\mathbb{C})}_{\fl}$, which will be denoted the same; in this case the underlying Hermitian line bundles are trivial and the prequantum connections are defined by global symplectic potentials (cf. \S~\ref{sec:bargmann_transform}).
The lifted $\mathcal{K}_{0}$-action is determined by~\eqref{eq:lifted_K_action} on the generators.
The action of an element $w \in W$ on $\mathcal{A}_{0}^{(\mathbb{C})}$ can be represented by a constant-valued gauge transformation, which pulls the Maurer-Cartan form back to $0$, showing that $W$ acts trivially on the fibres of $\mathcal{L}_{k} \to \mathcal{M}_{\fl}$ and $\mathcal{L}_{t}^{\mathbb{C}} \to \mathcal{M}_{\fl}^{\mathbb{C}}$.
On the other hand, the translation by an element $a \in \mathcal{T}_{0} \subset \mathcal{A}_{0}^{(\mathbb{C})}$ is represented by a gauge transformation $\gamma$ valued in $T \subset K$ such that $a = [\theta_{\gamma}]$.
In particular, $W_{\Sigma, t}(\gamma)$ vanishes since $T$ is abelian, and~\eqref{eq:lifted_K_action} reduces to
\begin{equation}
	\label{eq:cocycle}
	\Theta_{a, A, t} \coloneqq \Theta_{\gamma, A, t} = \exp \left( -\frac{i}{2} \omega_{t} (A, a) \right) \, .
\end{equation}

We then consider the Hilbert spaces $\operatorname{L}_{t}^{2,\mathbb{C}} = \operatorname{L}^{2}\bigl(\mathcal{M}_{\fl}^{\mathbb{C}},\mathcal{L}_{t}^{\mathbb{C}}\bigr)$ and $\operatorname{L}_{k}^{2} = \operatorname{L}^{2}(\mathcal{M}_{\fl},\mathcal{L}_{k})$ of $\operatorname{L}^2$-sections of the prequantum line bundles, with respect to the Liouville volume forms, and then the trivial Hilbert bundles with these fibres:
\begin{equation}
    \pmb{\operatorname{L}}_{t}^{2,\mathbb{C}} \coloneqq \operatorname{L}_{t}^{2,\mathbb{C}} \times \mathcal{T} \longrightarrow \mathcal{T} \, , \qquad \pmb{\operatorname{L}}_{k}^{2}\coloneqq \operatorname{L}_{k}^{2}\times \mathcal{T} \longrightarrow \mathcal{T} \, .
\end{equation} 

Analogous Hilbert spaces/bundles are defined for sections over the covering space $\mathcal{A}_{0}^{(\mathbb{C})} \twoheadrightarrow \mathcal{M}_{\fl}^{(\mathbb{C})}$.
The resulting Hilbert spaces are written $\widetilde{\operatorname{L}}_{t}^{2,\mathbb{C}} = \operatorname{L}^{2}\bigl(\mathcal{A}_{0}^{\mathbb{C}},\mathcal{L}_{t}^{\mathbb{C}}\bigr)$ and $\widetilde{\operatorname{L}}_{k}^{2} = \operatorname{L}^{2}(\mathcal{A}_{0},\mathcal{L}_{k})$, and the resulting trivial Hilbert bundles are 
\begin{equation}
	\widetilde{\pmb{\operatorname{L}}}_{t}^{2,\mathbb{C}} \coloneqq \widetilde{\operatorname{L}}_{t}^{2,\mathbb{C}} \times \mathcal{T} \longrightarrow \mathcal{T} \, , \qquad \widetilde{\pmb{\operatorname{L}}}_{k}^{2}\coloneqq \widetilde{\operatorname{L}}_{k}^{2}\times \mathcal{T} \longrightarrow \mathcal{T} \, .
\end{equation} 

\subsection{K\"{a}hler quantisation}
\label{sec:kaehler_quantisation}

Let further $\tau$ be a variable in $\mathcal{T}$, and let $I_{\tau,t}$ and $I_{\tau}$ be the complex structures of \S~\ref{sec:kaehler_polarisations_i}.
The $(0,1$)-part of the prequantum connections define holomorphic structures on $\mathcal{L}_{t}^{\mathbb{C}}$ and $\mathcal{L}$, and we consider as customary the Hilbert subspaces of holomorphic sections inside $\operatorname{L}_t^{2,\mathbb{C}}$ and $\operatorname{L}^2_k$, denoted $\mathcal{H}^{\mathbb{C}}_{\tau,t}$ and $\mathcal{H}_{\tau,k}$ respectively.

Finally we look at smooth $\mathcal{T}$-families of holomorphic sections.
In the compact case these are smooth maps $\varphi \colon \mathcal{M}_{\fl} \times \mathcal{T} \to \mathcal{L}_{k}$ whose fibrewise restriction $\eval[1]{\varphi}_{\mathcal{M}_{\fl} \times \Set{\tau}} \colon (\mathcal{M}_{\fl},I_{\tau}) \to \mathcal{L}_{k}$ is an $I_{\tau}$-holomorphic section for $\tau \in \mathcal{T}$, and analogously in the complexified case.
(Intuitively we consider fibrewise K\"{a}hler quantisation of the fibre bundle $\bm{\mathcal{M}}_{\fl,t}^{\mathbb{C}}$ and of its sub-bundle $\bm{\mathcal{M}}_{\fl,k}$.) 

Analogous Hilbert spaces and $\mathcal{T}$-families of holomorphic sections are defined for $\mathcal{A}_{0}^{(\mathbb{C})} \to \mathcal{M}_{\fl}^{(\mathbb{C})}$.
The resulting quantum Hilbert spaces are $\widetilde{\mathcal{H}}^{\mathbb{C}}_{\tau,t}$ and $\widetilde{\mathcal{H}}_{\tau,k}$.

\subsection{Real quantisation}
\label{sec:real_quantisation}

As we noted in Rem.~\ref{rem:transverse_compact_moduli_space}, $\mathcal{M}_{\fl}$ is a global transverse for $P_{\tau}$ in $\mathcal{M}_{\fl}^{\mathbb{C}}$.
This implies that any $P_{\tau}$-polarised section of $\mathcal{L}_{t}^{\mathbb{C}} \to \mathcal{M}_{\fl}^{\mathbb{C}}$ is completely determined by its restriction to $\mathcal{M}_{\fl}$.
On the other hand, the stabilizer in $\mathcal{K}_{0}$ of every point of $\mathcal{A}_{0}$ acts linearly on the corresponding leaf, so that its quotient is contractible.
Therefore, any section of $\mathcal{L}_{k} \to \mathcal{M}_{\fl}$ extends uniquely by parallel transport to a $P_{\tau}$-polarised one.
Using this ($\mathcal{T}$-dependent) identification, we let $\operatorname{L}^{2}_{k} = \operatorname{L}^{2}\bigl(\mathcal{M}_{\fl},\mathcal{L}_{k}\bigr)$ be the quantum space associated to $P_{\tau}$.

Then the bundle arising from the fibrewise real quantisation of $\mathcal{M}_{\fl}^{\mathbb{C}}$ is the trivial Hilbert bundle $\pmb{\operatorname{L}}^{2}_{k} = \operatorname{L}^{2}_{k} \times \mathcal{T} \to \mathcal{T}$. 
The same construction applies verbatim for the linear real polarisation ${P}_{\tau}$ on the covering space $\mathcal{A}_{0}^{\mathbb{C}} \to \mathcal{M}_{\fl}^{\mathbb{C}}$. 

\begin{rem*}
    The trivial bundle carries the trivial flat connection, but its trivialisation depends on the $\mathcal{T}$-dependent splitting $\T\mathcal{M}_{\fl}^{\mathbb{C}} \simeq \T\mathcal{M}_{\fl} \oplus P_{\tau}$: hence one needs to construct a \emph{canonical} projectively flat connection, as done in~\cite{witten_1991_quantization_of_chern_simons_gauge_theory_with_complex_gauge_group,andersen_gammelgaard_2014_the_hitchin_witten_connection_and_complex_quantum_chern_simons_theory}.	
\end{rem*}

In the following, we will often use the word "polarised" to mean with respect to $P_{\tau}$.
Polarised objects with respect to the K\"ahler polarisation will be referred to as holomorphic.

\section{Flat quantum connections}
\label{sec:quantum_connections}

\subsection{Complexified Hitchin connection}
\label{sec:complexified_hitchin_connection}

For $\tau \in \mathcal{T}$ denote $\operatorname{T}^{1,0} = \operatorname{T}^{1,0}\mathcal{M}_{\fl}^{\mathbb{C}}$ (resp. $\operatorname{T}_{1,0} = \operatorname{T}_{1,0}\mathcal{M}_{\fl}^{\mathbb{C}}$) the $I_{\tau,t}$-holomorphic cotangent bundle (resp. the $I_{\tau,t}$-holomorphic tangent bundle), and similarly for the antiholomorphic parts.
Set also $\operatorname{T}^{\mathbb{C}} \coloneqq \operatorname{T}^{*}\mathcal{M}_{\fl}^{\mathbb{C}} \otimes \mathbb{C}$ (resp. $\operatorname{T}_{\mathbb{C}} \coloneqq T\mathcal{M}_{\fl}^{\mathbb{C}} \otimes \mathbb{C}$) for the complexified cotangent bundle (resp. complexified tangent bundle).

If $V$ is a tangent vector on $\mathcal{T}$ the derivative $V \bigl[I_{\tau,t} \bigr]$ is a section of $\End (\operatorname{T}_{\mathbb{C}})$ swapping $\operatorname{T}_{1,0}$ and $\operatorname{T}_{0,1}$. 
Then using $\T_{\mathbb{C}} = \T_{1,0} \oplus \T_{0,1}$ decompose 
\begin{equation}
    \label{eq:decomposition_variation_complex_structure}
    V \bigl[I_{\tau,t} \bigr] = V \bigl[I_{\tau,t} \bigr]' + V \bigl[I_{\tau,t} \bigr]'' \, ,
\end{equation} 
where the former takes values inside $\T_{1,0}$.
Since $\omega_{t}$ is non-degenerate there exists a unique bi-vector field $\widetilde{G}^{\mathbb{C}} (V)$ such that $\widetilde{G}^{\mathbb{C}} (V) \cdot \omega_{t} = \abs{t} V[ I_{\tau , t}]$, with an analogous splitting 
\begin{equation}
    \widetilde{G}^{\mathbb{C}}(V) = G^{\mathbb{C}}(V) + \overline{G}^{\mathbb{C}}(V) \, ,
\end{equation} 
with $G^{\mathbb{C}}(V) \in \Omega^0\bigl(\mathcal{M}_{\fl}^{\mathbb{C}},\operatorname{T}_{1,0} \otimes \operatorname{T}_{1,0}\bigr)$ and $\overline{G}^{\mathbb{C}}(V) \in \Omega^0\bigl(\mathcal{M}_{\fl}^{\mathbb{C}},\operatorname{T}_{0,1} \otimes \operatorname{T}_{0,1}\bigr)$.

\begin{rem*}

If $\widetilde{g}^{\mathbb{C}}_{\tau}$ is the inverse of the K\"{a}hler metric, one has
$V[\widetilde{g}_{\tau}^{\mathbb{C}}] \cdot \omega_{t} = - \abs{t} V[ I_{\tau, t}]$, since $\omega$ is $\mathcal{T}$-independent; hence we may write $\widetilde{G}^{\mathbb{C}}(V) = -V\bigl[ \widetilde{g}^{\mathbb{C}}_{\tau} \bigr]$.
In particular $\widetilde{G}^{\mathbb{C}}(V)$, $G^{\mathbb{C}}(V)$ and $\overline{G}^{\mathbb{C}}(V)$ are symmetric tensors.
\end{rem*}

\begin{rem*}
	Since the Levi-Civita connection is independent of $\tau$, differentiating $\nabla g_{\tau}^{\mathbb{C}} = 0$ along $V$ shows that $\widetilde{G}^{\mathbb{C}}(V)$ is parallel, and therefore so are its two components.
	In particular $G^{\mathbb{C}}(V)$ is holomorphic---we say the family of complex structures $\Set{ I_{\tau,t} }_{\tau \in \mathcal{T}}$ is \emph{rigid}~\cite{andersen_2012_hitchin_connection_toeplitz_operators_and_symmetry_invariant_deformation_quantisation}.
\end{rem*}

Now we consider the Laplacian operator associated to the symmetric tensor $G^{\mathbb{C}}(V)$, i.e. formally $\Delta_{G^{\mathbb{C}}(V)} = \Tr \bigl(\nabla^{1,0}_t G^{\mathbb{C}}(V) \nabla^{1,0}_t\bigr)$, see op. cit.
Letting $V$ vary yields a 1-form $u^{\mathbb{C}} \coloneqq - \Delta_{G^{\mathbb{C}}(\bullet)}$ on $\mathcal{T}$, with values in differential operators acting on smooth sections of the prequantum line bundle.

\begin{thm}
\label{thm:complexified_hitchin_connection}

The connection $\nabla^{\mathbb{C}} \coloneqq \nabla^{\Tr} - \frac{1}{4 \lvert t \rvert } u^{\mathbb{C}}$ is flat and preserves holomorphicity.\footnote{Hence~\cite[Rem.~4.16]{rembado_2018_quantisation_of_moduli_spaces_and_connections} is vindicated.}
\end{thm}

In the proof we will use the following identities, valid for vector fields $V$ on $\mathcal{T}$:
\begin{equation}
    \label{eq:tensor_identities}
    \begin{split}
        &V\bigl[ \nabla_{t} \bigr] = V[\omega_t] = 0 = \nabla_{t}\omega_{t}= \nabla_{t} G^{\mathbb{C}}(V) \, , \\
        \bigl[ \nabla^{1,0}_{t},\nabla^{1,0}_{t}\bigr] = &F_{\nabla_t}^{2,0} = -i\omega_t^{2,0} = 0, \qquad \bigl[ \nabla^{0,1}_{t},\nabla^{1,0}_{t}\bigr] = F_{\nabla_t}^{1,1} = -i\omega_{t}\, . 
    \end{split}
\end{equation}

They come from the following facts: $\nabla_t$ and $\omega_t$ are independent of $\tau$, $\omega_t$ is a K\"{a}hler form, $G^{\mathbb{C}}(V)$ is parallel, and $F_{\nabla_t} = -i\omega_t$ is of bidegree $(1,1)$.

\begin{proof}

Choose a $\mathcal{T}$-family of holomorphic sections $\varphi$, a vector field $X$ on $\mathcal{M}_{\mathrm{fl}}^{\mathbb{C}}$, and a vector field $V$ on $\mathcal{T}$.
We will start by showing $\nabla^{0,1}_{t,X} \bigl( V[\varphi] \bigr) = \frac{1}{4 \lvert t \rvert } \nabla^{0,1}_{t,X} \bigl( \Delta_{G^{\mathbb{C}}(V)} \varphi \bigr)$.

Using $\nabla^{0,1}_{t} \varphi = 0$ and~\eqref{eq:tensor_identities} yields
\begin{equation}
		\nabla_{t,X}^{0,1} \Delta_{G^{\mathbb{C}}(V)} \varphi
		= \Tr \Bigl(\bigl[\nabla^{0,1}_{t,X},\nabla^{1,0}_{t}\bigr] G^{\mathbb{C}}(V) \nabla^{1,0} \varphi + \nabla^{1,0}_{t}G^{\mathbb{C}}(V)  \bigl[\nabla^{0,1}_{t,X},\nabla^{1,0}_{t}\bigr] \varphi\Bigr) \, . 
\end{equation} 
Since the Levi--Civita connection is flat, both commutators are controlled by the contraction $-i X \cdot \omega$.
Then the symmetry of $G^{\mathbb{C}}(V)$ and~\eqref{eq:tensor_identities} yield
\begin{equation}
        \nabla_{t,X}^{0,1} \Delta_{G^{\mathbb{C}}(V)} \varphi
        = -2i X \cdot \bigl(\omega_{t} \cdot G^{\mathbb{C}}(V) \bigr) \cdot \nabla^{1,0}_{t}\varphi
        = 2i \lvert t \rvert \bigl( V \bigl[ I_{\tau,t} \bigr] X \bigr) \cdot \nabla_{t}\varphi \, ,
\end{equation} 
where in the last passage we use that the antiholomorphic parts do not contribute.

For the other term, differentiating the identity $\nabla^{0,1}_{t,X}\varphi = 0$ along $V$ yields 
\begin{equation}
    0 = V\bigl[\nabla^{0,1}_{t,X}\bigr] \varphi + \nabla^{0,1}_{t,X}V[\varphi] = \frac{i}{2} \bigl(V \bigl[I_{\tau,t} \bigr] X \bigr) \cdot \nabla_{t}\varphi + \nabla^{0,1}_{t,X}V[\varphi] \, ,
\end{equation} 
using 
\begin{equation}
    \nabla^{0,1}_{t,X}= \frac{1}{2} \bigl( (\mathbb{1} + iI_{\tau,t}) X \bigr) \cdot \nabla_{t}\, , \qquad V \bigl[ \nabla^{0,1}_{t,X}\bigr] = \frac{i}{2} \bigl( V \bigl[ I_{\tau,t} \bigr] X \bigr) \cdot \nabla_{t}\, .
\end{equation} 

For the second statement, if $V'$ is a (local) vector field on $\mathcal{T}$ that commutes with $V$, the curvature reads 
\begin{equation}
	\label{eq:flatness_proof}
    \big\langle F_{\nabla^{\mathbb{C}}}, V \wedge V' \big\rangle = - \frac{1}{4 \lvert t \rvert} \Bigl( V'\bigl[ \Delta_{G^{\mathbb{C}}(V)} \bigr] - V \bigl[ \Delta_{G^{\mathbb{C}}(V')} \bigr] \Bigr) + \frac{1}{16 \lvert t \rvert^{2}} \bigl[\Delta_{G^{\mathbb{C}}(V)},\Delta_{G^{\mathbb{C}}(V')} \bigr] \, ,
\end{equation} 
and we will show both summands vanish.

For the former, since $\widetilde{G}(V) = -V \bigl[ \widetilde{g}^{\mathbb{C}}_{\tau} \bigr]$ it follows that $\Delta_{G(V)} = - V \bigl[ \Delta_{\widetilde{g}^{\mathbb{C}}_{\tau}} \bigr]$ as differential operators acting on holomorphic sections of $\mathcal{L}_{t}^{\mathbb{C}}$.
Hence 
\begin{equation}
    V'\bigl[ \Delta_{G^{\mathbb{C}}(V)} \bigr] - V \bigl[ \Delta_{G^{\mathbb{C}}(V')} \bigr] = - V'V \bigl[ \Delta_{\widetilde{g}^{\mathbb{C}}_{\tau}} \bigr] + VV' \bigl[ \Delta_{\widetilde{g}^{\mathbb{C}}_{\tau}} \bigr] = \bigl[V,V'\bigr] \bigl[ \Delta_{\widetilde{g}^{\mathbb{C}}_{\tau}} \bigr] = 0 \, .
\end{equation} 

For the rightmost term of~\eqref{eq:flatness_proof}, we may use the Leibnitz rule to expand the commutator $\bigl[ \nabla^{1,0}_{t}G^{\mathbb{C}}(V) \nabla^{1,0}_{t}, \nabla^{1,0}_{t}G^{\mathbb{C}}(V') \nabla^{1,0}_{t}\bigr]$, which vanishes because of the identities~\eqref{eq:tensor_identities}, and because the contractions with $G^{\mathbb{C}}(V)$ and $G^{\mathbb{C}}(V')$ commute.
Hence
\begin{equation}
    \bigl[ \Delta_{G^{\mathbb{C}}(V)},\Delta_{G^{\mathbb{C}}(V')} \bigr] = \Tr \bigl[ \nabla^{1,0}_{t} G^{\mathbb{C}}(V) \nabla^{1,0}_{t},\nabla^{1,0}_{t} G^{\mathbb{C}}(V') \nabla^{1,0}_{t} \bigr] = 0 \, .
\end{equation} 
\end{proof}

The connection of Thm.~\ref{thm:complexified_hitchin_connection} should be compared with the Hitchin connection~\cite{hitchin_1990_flat_connections_and_geometric_quantisation,axelrod_dellapietra_witten_1991_geometric_quantisation_of_chern_simons_gauge_theory, andersen_2012_hitchin_connection_toeplitz_operators_and_symmetry_invariant_deformation_quantisation}.
In fact, noting that the Ricci potential on $\mathcal{M}_{\fl}$ vanishes, Eq.~1 of~\cite[Thm.~1]{andersen_2012_hitchin_connection_toeplitz_operators_and_symmetry_invariant_deformation_quantisation} is formally analogous to $\nabla^{\mathbb{C}}$ up to replacing $t$ with $k$, while Thm.~\ref{thm:complexified_hitchin_connection} shows it enjoys the two key properties of the original Hitchin connection.
We thus refer to this object as the \emph{complexified} Hitchin connection.

\begin{rem}
\label{rem:mapping_class_group_invariance}
 
 The complexified Hitchin connection of Thm.~\ref{thm:complexified_hitchin_connection} is $\Gamma$-invariant, being based on the variation of the Laplace--Beltrami operator for the $\Gamma$-equivariant K\"{a}hler metric (cf.~\cite[Lem.~6]{andersen_2012_hitchin_connection_toeplitz_operators_and_symmetry_invariant_deformation_quantisation}).
\end{rem}

This construction can be carried out on $\mathcal{A}_0^{\mathbb{C}}$, producing a $\Gamma$-invariant flat connection acting on $\mathcal{T}$-families of holomorphic sections of ${\mathcal{L}}^{\mathbb{C}}_{t} \to \mathcal{A}_{0}^{\mathbb{C}}$, which we refer to as the \emph{lifted} complexified Hitchin connection, also denoted ${\nabla}^{\mathbb{C}} = \nabla^{\Tr} - \frac{1}{4 \abs{t}} {u}^{\mathbb{C}}$.

\subsection{Hitchin--Witten connection}
\label{sec:hitchin_witten_connection}

Analogously to the previous \S~\ref{sec:complexified_hitchin_connection} consider the Laplacians 
\begin{equation}
    \Delta_{G(V)} \coloneqq \Tr \bigl( \nabla^{1,0} G(V) \nabla^{1,0}), \qquad \Delta_{\overline{G}(V)} \coloneqq \Tr \bigl( \nabla^{0,1} \overline{G}(V) \nabla^{0,1} \bigr) \, ,
\end{equation} 
using the variation of $I_{\tau}$ on $\mathcal{M}_{\fl}$.
Then we have the connection
\begin{equation}
\label{eq:hitchin_witten_connection}
    \nabla^{\HW} = \nabla^{\Tr} - \frac{1}{2} u^{\HW}, \qquad \text{where} \qquad u^{\HW}(V) \coloneqq \frac{1}{\overline{t}} \Delta_{\overline{G}(V)} - \frac{1}{t} \Delta_{G(V)} \, ,
\end{equation}
acting on $\mathcal{T}$-families of smooth sections of $\mathcal{L}_k \to \mathcal{M}_{\fl}$.
Since $g_{\tau}$ is flat this is a particular instance of the connection studied in~\cite{witten_1991_quantization_of_chern_simons_gauge_theory_with_complex_gauge_group}, i.e. the genus-one analogue of \cite{andersen_gammelgaard_2014_the_hitchin_witten_connection_and_complex_quantum_chern_simons_theory} (which considers the higher-genus case).
We refer to it as the \emph{Hitchin--Witten} connection.

The tensor calculus developed in~\cite{andersen_gammelgaard_2014_the_hitchin_witten_connection_and_complex_quantum_chern_simons_theory} applies to the genus-one case as well, and can be used to deduce flatness, analogously to the proof of Thm.~\ref{thm:complexified_hitchin_connection} (see also~\cite{malusa_2018_geometric_quantisation_the_hitchin_witten_connection_and_quantum_operators_in_complex_chern_simons_theory,AM}).
What is more, Witten argues in~\cite{witten_1991_quantization_of_chern_simons_gauge_theory_with_complex_gauge_group} that the connection admits an explicit trivialisation, something that was further explored for semi-simple groups in~\cite{AM,andersen_malusa_2017_the_aj_conjecture_for_the_teichmueller_tqft} and exploited in~\cite{andersen_malusa_2017_the_aj_conjecture_for_the_teichmueller_tqft,AM19}.

\begin{rem}
    Analogously to Rem.~\ref{rem:mapping_class_group_invariance}, the Hitchin--Witten connection~\eqref{eq:hitchin_witten_connection} is invariant under the group of bundle automorphisms of $\pmb{\operatorname{L}}^{2}_{k} \to \mathcal{T}_{\Sigma}$ defined by the mapping class group. 
\end{rem}

This construction can be carried out on $\mathcal{A}_{0}$, producing a $\Gamma$-invariant flat connection acting on $\mathcal{T}$-families of smooth section of ${\mathcal{L}}_{K} \to \mathcal{A}_{0}$, which we refer to as the \emph{lifted} Hitchin--Witten connection, also denoted ${\nabla}^{\HW} = \nabla^{\Tr} - \frac{1}{2} {u}^{\HW}$.

\section{The Bargmann transform}
\label{sec:bargmann_transform}

In this section we shall recall the general facts about the geometric quantisation on $\mathbb{C}^{m}$, $m$ a positive integer, and the Bargmann transform.
In coordinates $z_{j} = p_{j} + i q_{j}$, the standard symplectic structure can be expressed as
\begin{equation}
	{\omega} = \sum_{j = 1}^{m} \dif p_{j} \wedge \dif q_{j} = \frac{i}{2} \sum_{j = 1}^{m} \dif z_{j} \wedge \dif \overline{z}_{j} \, .
\end{equation}
There is a unique pre-quantum line bundle $\mathcal{L}^{\mathbb{C}}_{\hslash}$, up to isomorphism, for every positive real parameter $\hslash$.
We will fix the trivialisation so the pre-quantum connection reads $\nabla_{h} = \dif - \frac{i}{\hslash} \alpha$, where $\alpha$ is the invariant symplectic potential
\begin{equation}
	{\alpha} = \frac{i}{4} \sum_{j = 1}^{\frac{m}{2}} \bigl( z_{j} \dif \overline{z}_{j} - \overline{z}_{j} \dif z_{j} \bigr) = \frac{1}{2} \sum_{j = 1}^{m} \bigl( p_{j} \dif q_{j} - q_{j} \dif p_{j} \bigr) \, .
\end{equation}
One easily checks that the smooth functions
\begin{equation}
\label{eq:frames}
	\sigma \coloneqq \left( \frac{1}{2 \pi \hslash} \right)^{\frac{m}{2}} \exp \left( - \frac{1}{4 \hslash} \lvert \bm{z} \rvert^{2} \right)	\, , \qquad \rho \coloneqq \exp \left( - \frac{i}{2 \hslash} \sum_{j = 1}^{m} p_{j} q_{j} \right) \, ,
\end{equation}
are polarised frames for the tautological Kähler structure and the real polarisation $P = \mathbb{R}^{m}$, respectively.

The Hilbert space $\widetilde{\mathcal{H}}^{\mathbb{C}}_{\hslash}$ from Kähler quantisation, consisting of $\operatorname{L}^{2}$ holomorphic sections of $\mathcal{L}_{\hslash}$, can be identified with that of holomorphic \emph{functions} with finite $\operatorname{L}^{2}$-norm with respect to the Gaussian measure $\sigma^{2}$.
The latter is called the Segal--Bargmann space~\cite{bargmann_1961_hilbert_space_of_analytic_functions_and_associated_integral_transform,segal_1963_mathematical_problems_of_relativistic_physics}.
We will use the notation $f$ for a function and $\varphi = f \sigma$ for the corresponding section, and use the two viewpoints at convenience.

Analogously, an element of the quantum Hilbert space $\widetilde{\mathcal{H}}_{P}$ arising from $P$ can be viewed as either a function $\psi$ of the variables $q_{j}$ alone or as the corresponding polarised section $\psi \rho$.
The intrinsic definition of the inner product uses half-forms, but up to appropriate natural choices it can be identified with the $\operatorname{L}^{2}$-pairing for functions on $Q \coloneqq i \mathbb{R}^{m}$.
We shall often abuse notation and call $\psi$ both objects; note that they agree on $Q$ since $\eval[1]{\rho}_{{Q}} \equiv 1$.

The two quantum Hilbert spaces are related by a non-degenerate pairing, which for $\psi$ in an appropriate dense subspace is the $\operatorname{L}^{2}$-pairing on $\mathbb{C}^{m}$.
This defines by duality a linear isomorphism $\mathcal{B} \colon \widetilde{\mathcal{H}}_{P} \to \widetilde{\mathcal{H}}^{\mathbb{C}}_{\hslash}$, whose inverse we will denote $\mathcal{B}'$, which can be written in integral form as
\begin{equation}
	\label{eq:bargmann_new}
	\bigl(\mathcal{B}(\psi)\bigr) (\bm{z}) = \int_{\mathbb{R}^{m}} \psi(\bm{q}') B(\bm{q}',\bm{z}) \dif \bm{q} \, ,
	\qquad
	\bigl(\mathcal{B}'(\varphi)\bigr)(\bm{q}) = \int_{\mathbb{C}^{m}} \varphi(\bm{z}') \overline{B(\bm{q},\bm{z}')} \dif \bm{z}' \, ,
\end{equation}
where $\dif \bm{q}'$ and $\dif \bm{z}'$ denote the respective volume forms and
\begin{equation}
	\label{eq:Barg_kernel}
	\begin{split}
		B(\bm{q}',\bm{z}) \coloneqq{}& \left(\frac{\abs{t}^{3}}{4\pi^{3}}\right)^{\frac{m}{4}} \exp\left(-\frac{\abs{t}}{4} \bigl(2 \abs{\bm{q}'}^{2} + 4i\bm{q}'\cdot\bm{z} - \bm{z}\cdot\bm{z} + \abs{\bm{z}}^{2}\bigr)\right)=\\
		={}& \left(\frac{\abs{t}^{3}}{4\pi^{3}}\right)^{\frac{m}{4}} \exp\left(-\frac{\abs{t}}{2}\abs{\bm{q}-\bm{q}'}^{2}\right) \exp\left(-\frac{i\abs{t}}{2} \bm{p} \cdot(2\bm{q}'-\bm{q})\right)\, .
	\end{split}
\end{equation}
We emphasise that, in this form, the output of the Bargmann transform is a holomorphic section, rather than a function.
In this normalisation, the Bargmann transform is a unitary isomorphism between the quantum Hilbert spaces. 
These formul\ae{} are equivalent to those of~\cite[Chap.~V, \S~7]{guillemin_sternberg_1977_geometric_asymptotics} (or~\cite[Eq.~1.4]{bargmann_1961_hilbert_space_of_analytic_functions_and_associated_integral_transform}), only that we insist in using an invariant symplectic potential and that we parametrise differently the complex coordinates.

In the following we will often consider the operators $a^{*}_{j} f \coloneqq z_{j} f$ and $a_{j} f \coloneqq 2 \hslash \frac{\partial f}{\partial z_{j}}$, which are mutually adjoint in $\widetilde{\mathcal{H}}^{\mathbb{C}}_{\hslash}$ and often referred to as the ladder operators.
We will later use that, if $\pi^{\widetilde{\mathcal{H}}}$ denotes the orthogonal projection of the space of all $\operatorname{L}^{2}$ functions to the closed subspace $\widetilde{\mathcal{H}}^{\mathbb{C}}_{\hslash}$, then
\begin{equation}
	\label{eq:proj_zbar}
	\pi^{\widetilde{\mathcal{H}}} (\overline{z}_{j}f) = 2 \hslash \frac{\partial f}{\partial z_{j}} \, .
\end{equation}

We shall use the following fundamental property of the Bargmann transform, expressing the fact that it identifies the two quantum Hilbert spaces as Fock representations.

\begin{prop}[cf.~\cite{bargmann_1961_hilbert_space_of_analytic_functions_and_associated_integral_transform}, \S~1.8.i]
\label{prop:ladder}
If $\psi$ is a smooth function with $\psi, q_{j}\psi, \frac{\partial \psi}{\partial q_{j}} \in \widetilde{\mathcal{H}}_{\hslash}$ for a fixed $j$, then $\mathcal{B}(\psi)$ lies in the domain of the operators $a_{j} \pm a_{j}^{*}$, and
\begin{equation}
\label{eq:commutation_bargmann}
	\begin{aligned}
		\mathcal{B} \left( q_{j} \psi \right) ={}& \frac{i}{2} (a_{j} - a^{*}_{j}) \mathcal{B}(\psi) \, , \\
		\mathcal{B} \left(\frac{\partial \psi}{\partial q_{j}} \right) ={}& \frac{i}{2\hslash} (a_{j} + a^{*}_{j}) \mathcal{B}(\psi) \, .
	\end{aligned}
\end{equation}
	
\end{prop}

\begin{rem}
\label{rem:general_linear_space}
	The setup described in this section applies to any abstract linear symplectic space with a Kähler and a real polarisation, the identification being obtained by choosing any orthonormal basis of the real Lagrangian.
\end{rem}

\section{Coordinates and frames}
\label{sec:coordinates_and_frames}

In this section we define local coordinates on the moduli spaces, and fix conventions for later use. 
The same discussion is presented in further detail for $K = \SU(2)$ and $K^{\mathbb{C}} = \SL(2,\mathbb{C})$ in~\cite{malusa_2018_geometric_quantisation_the_hitchin_witten_connection_and_quantum_operators_in_complex_chern_simons_theory,andersen_malusa_2017_the_aj_conjecture_for_the_teichmueller_tqft,rembado_2018_quantisation_of_moduli_spaces_and_connections}.

Consider on $\Sigma$ the coordinates $(x,y)$ induced by the identification $\Sigma \simeq \mathbb{R}^{2}/\mathbb{Z}^{2}$.
The choice of a basis $(T_{1} , \dotsc , T_{r})$ of $\mathfrak{t}$ induces global linear coordinates $\bm{w} = \bm{u} + i \bm{v}$ on $\mathcal{A}_0^{\mathbb{C}}$ via the identification $\mathcal{A}_{0}^{\mathbb{C}} \simeq H^{1} (\Sigma , \mathbb{R}) \otimes \mathfrak{t}^{\mathbb{C}}$; similarly, $\bm{u}$ defines coordinates on $\mathcal{A}_{0}$.
Having fixed coordinates on $\Sigma$ one can identify $\mathcal{A}_{0}^{(\mathbb{C})}$ with the space of $\mathfrak{t}^{(\mathbb{C})}$-valued $1$-forms with constant coefficients.
Assuming in addition that the basis $(T_{1} , \dotsc , T_{r})$ is $\langle \cdot,\cdot \rangle_{\mathfrak{k}^{\mathbb{C}}}$-orthonormal then ${\omega}^{\mathbb{C}} = \sum_{j = 1}^{r} \dif w_{j} \wedge \dif w_{r+j}$.
These coordinates however do not trivialise the additional structure induced by a choice of $\tau \in \mathcal{T}$, so we introduce new $\mathcal{T}$-dependent ones.

Given $\tau = \tau_{1} + i \tau_{2} \in \mathbb{H}$, the corresponding class of Kähler structures is represented by one with holomorphic coordinate $\zeta_{\tau} \coloneqq x + \tau y$ and Hodge $*$-operator
\begin{equation}
   	* \dif x ={} \frac{1}{\tau_{2}} \bigl( \tau_{1} \dif x + \abs{\tau}^{2} \dif y \bigr) \, , \qquad * \dif y ={} - \frac{1}{\tau_{2}} \bigl( \dif x + \tau_{1} \dif y \bigr) \, .
\end{equation}
It is immediate to check that the standard decomposition of a $\mathfrak{t}^{\mathbb{C}}$-valued connection form with constant coefficients trivially satisfies Hitchin's equations for this structure, meaning that its harmonic metric is the trivial one.
In turn, since harmonic forms with respect to this metric are precisely those with constant coefficients, the complex structure $I_{\tau}^{(\mathbb{C})}$ is represented, in the model of $\mathcal{A}_{0}^{(\mathbb{C})}$ just introduced above, by the trivial Hodge $*$-operator.
Note in particular that the resulting (hyper-)K\"ahler structure on $\mathcal{A}_{\tau}^{(\mathbb{C})}$ is linear, so the Levi-Civita connection is trivial as claimed earlier.

An orthonormal basis of the real polarisation ${P}_{\tau} \subseteq \mathcal{A}_{0}^{\mathbb{C}}$ (as a complex vector space with structure ${J}$) is given by the elements
\begin{equation}
\label{eq:frameP}
	X_{j} \coloneqq \frac{T_{j}}{\sqrt{2 \tau_{2}}} \dif \zeta_{\tau} \quad \text{for } j \in \Set{1, \dotsc, r} \, .
\end{equation}

Fix now $t = k + is$ with integer real part, thus selecting a K\"{a}hler structure $(\mathcal{A}_{0}^{\mathbb{C}},\omega_{t},I_{\tau,t})$ as in \S~\ref{sec:polarisations}. We then construct a real basis for $\mathcal{A}_{0}^{\mathbb{C}}$ by considering the vectors $X_{j}$ in~\eqref{eq:frameP} together with
\begin{equation}
\label{eq:frameall}
	X_{j + r} \coloneqq J X_{j} \quad \text{for } 1 \leq j \leq r , \qquad Y_{j} \coloneqq {I}_{\tau,t} X_{j} \quad \text{for } 1 \leq j \leq 2r \, .
\end{equation}
We will denote $(\bm{p} , \bm{q})$ the corresponding linear coordinates, with $\bm{p} = (p_{1} , \dotsc, p_{2r})$ and $\bm{q} = (q_{1} , \dotsc, q_{2r})$ corresponding to the $X_{j}$ and the $Y_{j}$'s, respectively.
We will call $\bm{z} = \bm{p} + i \bm{q}$ the corresponding $I_{t, \tau}$-holomorphic coordinates, and often write $A(\tau, \bm{p}, \bm{q})$ to denote the connection form corresponding to the parameters.

\begin{defn}
\label{def:trivial_covariant_derivative_tau}
	We denote by $\frac{\delta}{\delta \tau}$ the vector fields on $\mathcal{T}_{\Sigma} \times \mathcal{A}_{0}^{\mathbb{C}}$ given by 
	\begin{equation}
		\label{eq:deltaderivatives}
		\frac{\delta}{\delta \tau} = \frac{\partial}{\partial \tau} - \sum_{j = 1}^{2r} \left( \frac{\partial p_{j}}{\partial \tau} \frac{\partial}{\partial p_{j}} + \frac{\partial q_{j}}{\partial \tau} \frac{\partial}{\partial q_{j}} \right) =
		\frac{\partial}{\partial \tau} - \left( \sum_{j = 1}^{2r} \frac{\partial z_{j}}{\partial \tau} \frac{\partial}{\partial z_{j}} + \frac{\partial \overline{z}_{j}}{\partial \tau} \frac{\partial}{\partial \overline{z}_{j}} \right) \, .
	\end{equation}
\end{defn}

\begin{rem}
\label{rem:deltader}
	Note that differentiation along these vectors preserves the property of being polarised with respect to both polarisations, because for every $j \in \Set{1, \dotsc, 2r}$ they commute with $\frac{\partial}{\partial \overline{z}_{j}}$ and $\frac{\partial}{\partial p_{j}}$.
\end{rem}

\begin{defn}
	\label{def:operators}
	We define operators acting on smooth functions $\mathcal{A}_{0}^{\mathbb{C}} \to \mathbb{C}$:
	{\everymath={\displaystyle}
	\begin{equation}
		\begin{array}{ccc}
			M_{j} \psi \coloneqq (q_{j} + i q_{j+r}) \psi \, , 
			& \qquad &\mu_{j} f \coloneqq (z_{j} + i z_{j+r}) f \, ,
				\\[.5em]
			D_{j} \psi \coloneqq \frac{1}{\abs{t}} \left( \frac{\partial}{\partial q_{j}} + i \frac{\partial}{\partial q_{j+r}} \right) \psi \, ,
			& \qquad &
			\delta_{j} f \coloneqq \frac{2}{\abs{t}} \left( \frac{\partial}{\partial z_{j}} + i \frac{\partial}{\partial z_{j+r}} \right) f \, .
		\end{array}
	\end{equation}}
	
\end{defn}

The two operators in each column commute, and~\eqref{eq:commutation_bargmann} becomes
\begin{equation}
	\label{eq:transfoperators}
	\mathcal{B} \circ M_{j} = \frac{i}{2} (\delta_{j} - \mu_{j}) \circ \mathcal{B}
	\qquad \text{and} \qquad
	\mathcal{B} \circ D_{j} = \frac{i}{2} (\mu_{j} + \delta_{j}) \circ \mathcal{B} .
\end{equation}

The transition between the two coordinate systems $(\bm{p},\bm{q})$ and $(\bm{u},\bm{v})$ can be obtained from the identity $\dif \zeta_{\tau} = \dif x + \tau \dif y$, and from~\eqref{eq:frameP} and~\eqref{eq:frameall}.
In what follows we will only need the relations
\begin{equation}
	\label{eq:coordinates}
	\begin{gathered}
		q_{j} = \frac{1}{\lvert t \rvert  \sqrt{2 \tau_{2}}} \bigl( - (k \tau_{1} - s \tau_{2}) u_{j} + (k \tau_{2} + s \tau_{1}) v_{j} + k u_{j+r} - s v_{j+r} \bigr) \, , \\
		q_{j+r} = \frac{1}{\lvert t \rvert  \sqrt{2 \tau_{2}}} \bigl( (k \tau_{2} + s \tau_{1}) u_{j} + (k \tau_{1} - s \tau_{2}) v_{j} - s u_{j+r} - k v_{j+r} \bigr) \, ,
	\end{gathered}
\end{equation}
and the inverse relations show that $\mathcal{A}_{0} \subseteq \mathcal{A}_{0}^{\mathbb{C}}$ is expressed in coordinates $(\bm{p}, \bm{q})$ by
\begin{equation}
	\label{eq:A0pq}
	p_{j} =  \frac{s}{\abs{t}} q_{j} + \frac{k}{\abs{t}} q_{j+r} \, , 
	\qquad
	p_{j+r} = - \frac{k}{\abs{t}} q_{j} + \frac{s}{\abs{t}} q_{j+r} \, .
\end{equation}

We are in the situation of Rem.~\ref{rem:general_linear_space}.
Our setting corresponds to the symplectic form $\omega_{t} /\abs{t}$ for the quantum parameter $\hslash = 1/\abs{t}$; we then have frames $\sigma_{\tau}$ and $\rho_{\tau}$ as well as a Bargmann transform $\mathcal{B}_{\tau}$ for each $\tau$.

\subsection{Variations over Teichm\"{u}ller space}

Differentiating~\eqref{eq:coordinates} yields
\begin{equation}
	\label{eq:varcoord}
	\begin{gathered}
		\frac{\partial q_j}{\partial \tau} = 
			- \frac{1}{4 \tau_{2}} q_{j+r} - \frac{t}{4 \tau_{2} \lvert t \rvert } (p_{j} + i p_{j+r}) \, , \qquad
		\frac{\partial q_{j+r}}{\partial \tau} =
			\frac{1}{4 \tau_{2}} q_{j} - \frac{it}{4 \tau_{2} \lvert t \rvert } (p_{j} + i p_{j+r})
	\end{gathered}
\end{equation}
and similarly for the complex coordinates and variations in $\overline{\tau}$.

\begin{defn}
	\label{def:mathcalXj}
	For $j \in \Set{1, \dotsc, r}$ we set
	\begin{equation}
		\mathcal{X}_{j} \coloneqq \frac{1}{\sqrt{2 \tau_{2}}} \left( \frac{\partial}{\partial u_{j}} + \tau \frac{\partial}{\partial u_{j+r}} \right) \in \mathcal{A}_{0} \otimes_{\mathbb{R}} \mathbb{C} \subseteq \mathcal{A}_{0}^{\mathbb{C}} \otimes_{\mathbb{R}} \mathbb{C} \, .
	\end{equation}
\end{defn}

\begin{rem*}
	The above are defined formally in the same way as the vectors $X_{j}$ (cf.~\eqref{eq:frameP}), except they are thought of as \emph{complex} objects tangent to $\mathcal{A}_{0}$ rather than \emph{real} objects tangent to $\mathcal{A}_{0}^{\mathbb{C}}$.
	In particular they are anti-holomorphic and $I_{\tau} \mathcal{X}_{j} = - \mathcal{X}_{j}$.
\end{rem*}

\begin{lem}
\label{lem:variation_inverse_metric}

If $\widetilde{g}_{\tau}$ denotes the inverse of $g_{\tau}$, then
\begin{equation}
	\widetilde{G} \left( \frac{\partial}{\partial \tau} \right) = - \frac{ \partial \widetilde{g}_{\tau}}{\partial \tau} = - \frac{i}{\tau_{2}} \sum_{j = 1}^{r} \overline{\mathcal{X}}_{j} \otimes \overline{\mathcal{X}}_{j} \, ,
		\qquad
	\widetilde{G} \left( \frac{\partial}{\partial \overline{\tau}} \right) = - \frac{ \partial \widetilde{g}_{\tau}}{\partial \overline{\tau}} = \frac{i}{\tau_{2}} \sum_{j = 1}^{r} \mathcal{X}_{j} \otimes \mathcal{X}_{j} \, .
\end{equation}
\end{lem}

\begin{proof}
    This is proven in~\cite{andersen_malusa_2017_the_aj_conjecture_for_the_teichmueller_tqft,malusa_2018_geometric_quantisation_the_hitchin_witten_connection_and_quantum_operators_in_complex_chern_simons_theory} for $K = \SU(2)$. The general case follows, since $\mathcal{A}_{0}$ can be decomposed as an orthogonal direct sum of $r$ copies of the rank-one case.
\end{proof} 

\begin{cor}
	\label{cor:varg}
	The derivatives of $g_{\tau}$ along $\tau$ and $\overline{\tau}$ read
	\begin{equation}
		\begin{gathered}
			\frac{\partial g_{\tau}}{\partial \tau} (A,B) = - \frac{i}{\tau_{2}} \sum_{j = 1}^{r} g \left( \overline{\mathcal{X}}_{j} , A \right) g \left( \overline{\mathcal{X}}_{j} , B \right) \, , \\
			\frac{\partial g_{\tau}}{\partial \overline{\tau}} (A,B) = \frac{i}{\tau_{2}} \sum_{j = 1}^{r} g \left( {\mathcal{X}}_{j} , A \right) g \left( {\mathcal{X}}_{j} , B \right) \, .
		\end{gathered}
	\end{equation}
\end{cor}

\begin{proof}
	By the usual formula for the derivative of the inverse matrix, we have that
	\begin{equation}
		\frac{\partial g_{\tau}}{\partial \tau} = - g_{\tau} \cdot \frac{\partial \widetilde{g}_{\tau}}{\partial \tau} \cdot g_{\tau} = g_{\tau} \cdot \widetilde{G} \left( \frac{\partial}{\partial \tau} \right) \cdot g_{\tau} \, ,
	\end{equation}
	and the result follows.
	The derivative in $\overline{\tau}$ is obtained the same way.
\end{proof}

From the formula, combined with the fact that the Levi-Civita connection of $g_{\tau}$ is trivial, one deduces the following.

\begin{cor}
	The covariant derivative with respect to the Hitchin--Witten connection is given by
	\begin{equation}
		\label{eq:explicitHWC}
		\nabla^{\HW}_{\tau} = \frac{\partial}{\partial \tau} - \frac{i}{2 t \tau_{2}} \sum_{j = 1}^{r} \nabla_{\overline{\mathcal{X}}_{j}} \nabla_{\overline{\mathcal{X}}_{j}} \, ,
		\qquad \text{and} \qquad
		\nabla^{\HW}_{\overline{\tau}} = \frac{\partial}{\partial \overline{\tau}} - \frac{i}{2 \overline{t} \tau_{2}} \sum_{j = 1}^{r} \nabla_{\mathcal{X}_{j}} \nabla_{\mathcal{X}_{j}} \, .
	\end{equation}
\end{cor}

\begin{lem}
	\label{lem:G}
	The symmetric tensor $G^{\mathbb{C}}$ is determined by the identities
	\begin{equation}
		\begin{gathered}
			G^{\mathbb{C}} \left(\frac{\partial}{\partial \tau} \right) = - \frac{i t}{\tau_{2} \abs{t}} \sum_{j = 1}^{r} \left( \frac{\partial}{\partial z_{j}} + i \frac{\partial}{\partial z_{j+r}} \right) \otimes \left( \frac{\partial}{\partial z_{j}} + i \frac{\partial}{\partial z_{j+r}} \right) \, , \\
			G^{\mathbb{C}} \left(\frac{\partial}{\partial \overline{\tau}} \right) = - \frac{i \overline{t}}{\tau_{2} \abs{t}} \sum_{j = 1}^{r} \left( \frac{\partial}{\partial z_{j}} - i \frac{\partial}{\partial z_{j+r}} \right) \otimes \left( \frac{\partial}{\partial z_{j}} - i \frac{\partial}{\partial z_{j+r}} \right) \, .
		\end{gathered}
	\end{equation}
\end{lem}

\begin{proof}
	Since the decomposition $\mathcal{A}_{0}^{\mathbb{C}} = \mathcal{A}_{0} \oplus {J} \mathcal{A}_{0}$ is orthogonal, and since ${J}$ is an isometry, the metric $g^{\mathbb{C}}_{\tau}$ splits as the sum of two blocks $g_{\tau}$ and ${J}^{*} g_{\tau} = {J} \cdot g_{\tau} \cdot {J}$.
	Correspondingly, its inverse also splits as the sum
	\begin{equation}
		\widetilde{g}_\tau^{\mathbb{C}} = \widetilde{g}_{\tau} \oplus \bigl( {J}^{-1} \cdot \widetilde{g}_{\tau} \cdot {J}^{-1} \bigr) = \widetilde{g}_{\tau} \oplus \bigl( {J} \cdot \widetilde{g}_{\tau} \cdot {J} \bigr) \, .
	\end{equation}
	Since both ${J}$ and the splitting of $\mathcal{A}_{0}^{\mathbb{C}}$ are independent of the Teichmüller parameter, the derivatives of $\widetilde{g}_{\tau}$ with respect to $\tau$ and $\overline{\tau}$ also decompose in a similar way, whence
	\begin{equation}
		\begin{gathered}
			\frac{\partial \widetilde{g}^{\mathbb{C}}_{\tau}}{\partial \tau} = \frac{i}{\tau_{2}} \sum_{j = 1}^{r} \left( \overline{\mathcal{X}}_{j} \otimes \overline{\mathcal{X}}_{j} + {J} \overline{\mathcal{X}}_{j} \otimes {J} \overline{\mathcal{X}}_{j} \right)	\, , \qquad
			\frac{\partial \widetilde{g}^{\mathbb{C}}_{\tau}}{\partial \overline{\tau}} = - \frac{i}{\tau_{2}} \sum_{j = 1}^{r} \left( {\mathcal{X}}_{j} \otimes {\mathcal{X}}_{j} + {J} {\mathcal{X}}_{j} \otimes {J} {\mathcal{X}}_{j} \right) \, .
		\end{gathered}
	\end{equation}
	Moreover a direct computation shows that 
	{\everymath={\displaystyle}
	\begin{equation}
		\begin{array}{ccc}
			\dif z_{j} \bigl(\overline{\mathcal{X}}_{l}\bigr) = \frac{\delta_{jl}}{2} \left( 1 + \frac{t}{\abs{t}} \right) ,
			& \qquad &
			\dif z_{j+r} \bigl(\overline{\mathcal{X}}_{l}\bigr) = \frac{i \delta_{jl}}{2} \left( 1 + \frac{t}{\abs{t}} \right) , \\
			\dif z_{j} \bigl({J} \overline{\mathcal{X}}_{l}\bigr) = - \frac{i \delta_{jl}}{2} \left( 1 - \frac{t}{\abs{t}} \right) ,
			& \qquad &
			\dif z_{j+r} \bigl({J} \overline{\mathcal{X}}_{l}\bigr) = \frac{1 \delta_{jl}}{2} \left( 1 - \frac{t}{\abs{t}} \right) \, . 
		\end{array}
	\end{equation}}
	Therefore the components of $\overline{\mathcal{X}}_{j}$ and ${J}\overline{\mathcal{X}}_{j}$ of type $(1,0)$ with respect to ${I}_{\tau,t}$ are
	\begin{equation}
		\begin{gathered}
			\overline{\mathcal{X}}_{j} ' = \frac{1}{2} \left( 1 + \frac{t}{\abs{t}} \right) \left( \frac{\partial}{\partial z_{j}} + i \frac{\partial}{\partial z_{j+r}} \right) ,	\qquad
			({J} \overline{\mathcal{X}}_{j} )' = - \frac{i}{2} \left( 1 - \frac{t}{\abs{t}} \right) \left( \frac{\partial}{\partial z_{j}} + i \frac{\partial}{\partial z_{j+r}} \right) ,
		\end{gathered}
	\end{equation}
	respectively. 
	We conclude that
	\begin{equation}
		\overline{\mathcal{X}}_{j} \otimes \overline{\mathcal{X}}_{j} + {J} \overline{\mathcal{X}}_{j} \otimes {J} \overline{\mathcal{X}}_{j} 
		= \frac{t}{\abs{t}} \left( \frac{\partial}{\partial z_{j}} + i \frac{\partial}{\partial z_{j+r}} \right)^{\otimes 2} \, ,
	\end{equation}
	and the first identity in the statement is proven. The second is analogous.
\end{proof}

\begin{cor}
	\label{lem:CHCexplicit}
	The covariant derivative with respect to the complexified Hitchin connection is given by
	\begin{equation}
		\begin{gathered}
			\nabla^{\mathbb{C}}_{\tau} =
			\frac{\partial}{\partial \tau} - \frac{i}{4 \tau_{2} \overline{t}} \sum_{j = 1}^{r} \left( \nabla_{z_{j}} + i \nabla_{z_{j+r}} \right)^{2} \, , \qquad
			\nabla^{\mathbb{C}}_{\overline{\tau}} = \frac{\partial}{\partial \overline{\tau}} - \frac{i}{4 \tau_{2} t} \sum_{j = 1}^{r} \left( \nabla_{z_{j}} - i \nabla_{z_{j+r}} \right)^{2} \, .
		\end{gathered}
	\end{equation}
\end{cor}

\section{Identification of the connections on the covering spaces}
\label{sec:linconn}

\subsection{The \texorpdfstring{$\operatorname{L}^{2}$}{L-2}-connection}
\label{sec:l^2_connection}

Let $U \subset \mathcal{T}$ be an open subset and $f \colon \mathcal{A}_{0}^{\mathbb{C}} \times U \to \mathbb{C}$ a smooth function whose fibrewise restriction
\begin{equation}
	\eval[1]{f}_{\tau} \coloneqq \eval[1]{f}_{\mathcal{A}_{0}^{\mathbb{C}} \times \set{\tau}} \colon \mathcal{A}_{0}^{\mathbb{C}} \to \mathbb{C} 
\end{equation}
lies in $\widetilde{\mathcal{H}}^{\mathbb{C}}_{\tau,t}$ for every $\tau \in U$.
Let $V$ be a tangent vector on $\mathcal{T}$ and assume that $V[f \sigma_{\tau}]$ is $\operatorname{L}^2$.

\begin{defn}[$\operatorname{L}^{2}$-connection]
	The covariant derivative of $\varphi \coloneqq f \sigma_{\tau}$ along $V$ with respect to the $\operatorname{L}^{2}$-connection is
	\begin{equation}
		\nabla^{\operatorname{L}^{2}}_{V} \varphi \coloneqq \pi^{\widetilde{\mathcal{H}}} \left( V[\varphi] \right) \, .
	\end{equation}
\end{defn}

\begin{prop}[\cite{axelrod_dellapietra_witten_1991_geometric_quantisation_of_chern_simons_gauge_theory}, \S~1a]
	Suppose that $\varphi = f \sigma$ and $V$ are as above, and that moreover $\varphi$ lies in the domain of the ladder operators and their two-fold compositions.
	Then
	\begin{equation}
		\nabla^{\operatorname{L}^{2}}_{V} \varphi = {\nabla}^{\mathbb{C}}_{V} \varphi \, .
	\end{equation}
\end{prop}

\begin{proof}
    Following the proof of ~\cite[p.~311]{andersen_2012_hitchin_connection_toeplitz_operators_and_symmetry_invariant_deformation_quantisation}, for $\tau \in \mathcal{T}$ and $X$ a vector field of type $(1,0)$ on $\mathcal{A}_{0}^{\mathbb{C}}$, the adjoint operator of $\nabla_{X}$ on $\operatorname{L}^{2}(\mathcal{A}_{0}^{\mathbb{C}}, \mathcal{L}_{t}^{\mathbb{C}})$ is $\left(\nabla_{X}\right)^{*} = \operatorname{div} X - \nabla_{\overline{X}}$.
    Since the vector fields $\overline{\mathcal{X}}_{j}$ are constant and of type $(1,0)$, it follows that
    \begin{equation}
        \braket{\nabla_{\overline{\mathcal{X}}_{j}} \nabla_{\overline{\mathcal{X}}_{j}} \varphi | \varphi'} = \braket{\nabla_{\overline{\mathcal{X}}_{j}} \varphi | \nabla_{\overline{\mathcal{X}}_{j}}^{*} \varphi'} = \braket{\nabla_{\overline{\mathcal{X}}_{j}} \varphi | \nabla_{\mathcal{X}_{j}} \varphi'} = 0
    \end{equation}
    for all $\varphi' \in \widetilde{\mathcal{H}}^{\mathbb{C}}_{\tau,t}$ and $j \in \Set{1,\dotsc,r}$, and similarly for $J \overline{\mathcal{X}}_{j}$.
	Hence $\pi_{\tau}^{\widetilde{\mathcal{H}}} \bigl( u^{\mathbb{C}} (V) \varphi \bigr) = 0$ which yields
    \begin{equation}
        \nabla_{V}^{\mathbb{C}} \varphi = \pi^{\widetilde{\mathcal{H}}} \left( \nabla_{V}^{\mathbb{C}} \varphi \right) = \pi^{\widetilde{\mathcal{H}}} \left( V [\varphi] \right) = \nabla_{V}^{\operatorname{L}^{2}} \varphi \, .
    \end{equation}
\end{proof}

In the next \S~\ref{sec:conjugation_hitchin_witten} we will use the explicit local expression of the $\operatorname{L}^{2}$-connection.
This can be obtained by writing the derivative $\frac{\partial \varphi}{\partial \tau}$ using the vector field $\frac{\delta}{\delta \tau}$ from~\eqref{eq:deltaderivatives} and then combining~\eqref{eq:varcoord} with~\eqref{eq:proj_zbar}.
This leads to
\begin{equation}
		\label{eq:L2explicit}
	 		\nabla^{\operatorname{L}^2}_{\tau} \varphi = \frac{\delta f}{\delta \tau} \sigma_{\tau} - \frac{i}{16 \tau_{2}} \sum_{j = 1}^{r} \left( t \delta_{j}^{2} f + t \mu_{j}^{2} f - 4i z_{j+r} \frac{\partial f}{\partial z_{j}} + 4i z_{j} \frac{\partial f}{\partial z_{j}} \right) \sigma_{\tau} \, .
\end{equation}

\subsection{Conjugation of the Hitchin--Witten connection}
\label{sec:conjugation_hitchin_witten}

Consider a $\mathcal{T}$-family $\psi$ of functions on $\mathcal{A}_{0}$, corresponding to a family of sections $\psi \rho_{\tau}$ on $\mathcal{A}_{0}^{\mathbb{C}}$.
We will show its (lifted) Hitchin--Witten covariant derivative along $\frac{\partial}{\partial \tau}$ is
\begin{equation}
	\label{eq:HWexplicit}
	{\nabla}^{\HW}_{\tau} ( \psi \rho_{\tau}) = \frac{\partial \psi}{\partial \tau} \rho_{\tau} + \frac{it}{8 \overline{t} \tau_{2}} \sum_{j = 1}^{r} \left( \overline{t} D_{j}^{2} \psi - 2 \abs{t} M_{j} D_{j} \psi + \overline{t} M_{j}^{2} \psi \right) \rho_{\tau} \, .
\end{equation}

Next we study the polarised extension of the right-hand side to $\mathcal{A}_{0}^{\mathbb{C}}$, and then restrict it to ${Q}_{\tau}$, proving the result is the section
\begin{equation}
	\label{eq:HWextQ}
	\begin{gathered}
		\eval[3]{ \Ext \left( {\nabla}^{\HW}_{\tau} (\psi \rho) \right) }_{{Q}_{\tau}} = \frac{\partial \psi}{\partial \tau} + \frac{it}{8 \tau_{2}} \sum_{j = 1}^{r} \left( D_{j}^{2} + M_{j}^{2} \right) \psi \, ,
	\end{gathered}
\end{equation}
where $\Ext$ takes polarised extensions.

Finally we study the Bargmann transform of~\eqref{eq:HWextQ}, assuming that each summand is $\operatorname{L}^2$
and that $\psi$ is regular enough so that 
\begin{equation}
\label{eq:bargdercommute}
    \mathcal{B}_{\tau} \left( \frac{\delta \psi}{\delta \tau} \right) = \frac{\delta \mathcal{B}_{\tau} (\psi)}{\delta \tau}
    \qquad \text{and} \qquad
    \mathcal{B}_{\tau} \left( \frac{\delta \psi}{\delta \overline{\tau}} \right) = \frac{\delta \mathcal{B}_{\tau} (\psi)}{\delta \overline{\tau}} \, .
\end{equation}

Then we obtain the following reformulation of Thm.~\ref{thm:l^2=HW}.

\begin{thm}
	\label{thm:l^2=HWfull}
	Let $U \subseteq \mathcal{T}$ be open, and $\psi \colon U \times \mathcal{A}_{0}^{\mathbb{C}} \to \mathbb{C}$ a smooth family of polarised functions satisfying~\eqref{eq:bargdercommute} such that $\eval[1]{\psi}_{{Q}_{\tau}}$ lies in the domain of all two-fold compositions of ladder operators.
	Then
	\begin{equation}
		\begin{gathered}
			\mathcal{B}_{\tau} \left( \eval[2] {\nabla^{\HW}_{\tau} (\psi \rho)}_{{Q}_{\tau}} \right) = \nabla^{\operatorname{L}^{2}}_{\tau} \left( \mathcal{B}_{\tau} \left( \eval[1]{\psi}_{{Q}_{\tau}} \right) \right) \, , \\
			\mathcal{B}_{\tau} \left( \eval[2]{ {\nabla}^{\HW}_{\overline{\tau}} (\psi \rho) }_{{Q}_{\tau}} \right) = \nabla^{\operatorname{L}^{2}}_{\overline{\tau}} \left( \mathcal{B}_{\tau} \left( \eval[1]{\psi}_{{Q}_{\tau}} \right) \right) \, .
		\end{gathered}
	\end{equation}
\end{thm}

In the proof we will use 
	\begin{equation}
		\label{eq:projXj}
		\begin{gathered}
			\pi_{{Q}_{\tau}} \mathcal{X}_{j} = \frac{i \overline{t}}{2 \abs{t}} \left( \frac{\partial}{\partial q_{j}} - i \frac{\partial}{\partial q_{j+r}} \right) \, ,
			\qquad  \qquad
			\pi_{{Q}_{\tau}} \overline{\mathcal{X}}_{j} = - \frac{it}{2 \abs{t}} \left( \frac{\partial}{\partial q_{j}} + i \frac{\partial}{\partial q_{j+r}} \right) \, , \\
			g_{\tau} (A , \mathcal{X}_{j}) = \frac{i \overline{t}}{ \abs{t}} \left( q_{j} - i q_{j+r} \right) 
			\qquad \qquad
			g_{\tau} (A , \overline{\mathcal{X}}_{j}) = - \frac{i t}{ \abs{t}} \left( q_{j} + i q_{j+r} \right) \, ,
		\end{gathered}
	\end{equation}
	for $j \in \Set{1,\dotsc,r}$ and $A \in \mathcal{A}_{0}$, obtained from Def.~\ref{def:mathcalXj},~\eqref{eq:projPQ}, and~\eqref{eq:coordinates}.
	
\begin{proof}
	
	It is enough to verify the statement for the derivative in $\tau$. We start by proving~\eqref{eq:HWexplicit}.
	Using~\eqref{eq:explicitHWC} we have
	\begin{equation}
    \label{eq:HWCexpand1}
		{\nabla}^{\HW}_{\tau} (\psi \rho_{\tau}) = \frac{\partial \psi}{\partial \tau} \rho_{\tau} + \psi \frac{\partial \rho_{\tau}}{\partial \tau} - \frac{i}{2 t \tau_{2}} \sum_{j = 1}^{r} \nabla_{\overline{\mathcal{X}}_{j}} \nabla_{\overline{\mathcal{X}}_{j}} (\psi \rho_{\tau}) \, .
	\end{equation}
	
	To expand $\rho_{\tau}$ note $\abs{t} \bm{p} \cdot \bm{q}$ is the $\omega_{t}$-pairing of the projections of a vector onto $P_{\tau}$ and $Q_{\tau}$.
	For $A \in \mathcal{A}_{0}$ this can be written $\frac{s}{2} g_{\tau} (A,A)$, using~\eqref{eq:projPQ}, hence
	\begin{equation}
		\psi \frac{\partial \rho_{\tau}}{\partial \tau} = \frac{s t}{4 \tau_{2} \overline{t}} \sum_{j = 1}^{r} M_{j}^{2} \psi \rho_{\tau} \, ,
	\end{equation}
	by Cor.~\ref{cor:varg} and~\eqref{eq:projXj}.
	Similarly using $\eval[1]{\omega_{t}}_{\mathcal{A}_{0}} = k \omega$, and that $\overline{\mathcal{X}}_{j}$ is of type $(1,0)$ for $I_{\tau}$, one finds
	\begin{equation}
		\nabla_{\overline{\mathcal{X}}_{j}} \rho_{\tau} = - \frac{t}{2} g_{\tau} (A, \overline{\mathcal{X}}_{j}) \rho_{\tau}
		\qquad \text{and} \qquad
		\nabla_{\overline{\mathcal{X}}_{j}} \nabla_{\overline{\mathcal{X}}_{j}} \rho_{\tau} = \frac{t^{2}}{4} \left( g_{\tau} (A , \overline{\mathcal{X}}_{j} ) \right)^{2} \rho_{\tau} \, .
	\end{equation}
	Moreover, since $\psi$ is a polarised \emph{function}, so are all its derivatives, so using Def.~\ref{def:mathcalXj} and~\eqref{eq:projXj} one has 
	\begin{equation}
		\overline{\mathcal{X}}_{j} [ \psi ] = - \frac{it}{2} D_{j} \psi
		\qquad \text{and} \qquad
		\overline{\mathcal{X}}_{j} \left[ \overline{\mathcal{X}}_{j} [ \psi ] \right] = - \frac{t^{2}}{4} D_{j}^{2} \psi \, .
	\end{equation}
	Combining these relations and expanding the second-order operator in~\eqref{eq:HWCexpand1} yields
	\begin{equation}
			{\nabla}^{\HW}_{\tau} (\psi \rho_{\tau})
			= \frac{\partial \psi}{\partial \tau} \rho_{\tau} + \frac{it}{8 \overline{t} \tau_{2}} \sum_{j = 1}^{r} \left( \left( \overline{t} D_{j}^{2} - 2 \abs{t} M_{j} D_{j} + \overline{t} M_{j}^{2} \right) \psi \right) \rho_{\tau} \, ,
	\end{equation}
	which is equivalent to~\eqref{eq:HWexplicit}, as desired.
	
	Next we study the polarised extension of the above and establish~\eqref{eq:HWextQ}.
	In fact, each individual term in~\eqref{eq:HWexplicit} is a polarised object (restricted to $\mathcal{A}_{0}$), except for the term $\frac{\partial \psi}{\partial \tau}$, which can be expanded as
	\begin{equation}
    \label{eq:partialvsdelta}
        \begin{split}
            \frac{\partial \psi}{\partial \tau} \!=\! \frac{\delta \psi}{\delta \tau}
            \!-\! \frac{1}{4 \tau_{2}} \!\sum_{j = 1}^{r} \! \left( \! \left( \! q_{j+r} + \frac{t}{\abs{t}} (p_{j} + i p_{j+r}) \! \right) \! \frac{\partial \psi}{\partial q_{j}} \!-\! \left( \! q_{j} - \frac{it}{\abs{t}} (p_{j} + i p_{j+r}) \! \right) \! \frac{\partial \psi}{\partial q_{j+r}} \! \right) ,
        \end{split}
	\end{equation}
	using~\eqref{eq:deltaderivatives} and~\eqref{eq:varcoord} and the polarisation condition on $\psi$.
	By Rem.~\ref{rem:deltader}, the derivative $\frac{\delta \psi}{\delta \tau}$ is polarised, while using~\eqref{eq:A0pq} we can re-write
	\begin{equation}
		\eval[3]{ \frac{\partial \psi}{\partial \tau} }_{\mathcal{A}_{0}} = \frac{\delta \psi}{\delta \tau} - \frac{1}{4 \tau_{2}} \sum_{j = 1}^{r} \left( q_{j+r} \frac{\partial \psi}{\partial q_{j}} - q_{j} \frac{\partial \psi}{\partial q_{j+r}} - \frac{i t^{2}}{\abs{t}} M_{j} D_{j} \psi \right) \, .
	\end{equation}
	The right-hand side expresses a polarised function on $\mathcal{A}_{0}$ and therefore
	\begin{equation}
		\eval[4]{ \Ext \left( \eval[3]{ \frac{\partial \psi}{\partial \tau} }_{\mathcal{A}_{0}} \rho_{\tau} \right) }_{{Q}_{\tau}} = \eval[3]{ \frac{\partial \psi}{\partial \tau} }_{{Q}_{\tau}} + \frac{i t^{2}}{4 \abs{t} \tau_{2}} \sum_{j = 1}^{r} M_{j} D_{j} \eval[1]{ \psi }_{{Q}_{\tau}} \, .
	\end{equation}
	Combined with~\eqref{eq:HWexplicit}, this gives~\eqref{eq:HWextQ}.
	We can now apply the Bargmann transform to each term, using Prop.~\ref{prop:ladder} and condition~\eqref{eq:bargdercommute}, which yields
	\begin{equation}
		\begin{split}
			\mathcal{B}_{\tau} & \left( \eval[2]{ \Ext \left( {\nabla}^{\HW}_{\tau} (\psi \rho) \right) }_{{Q}_{\tau}} \right) = \\
			={}& \frac{ \delta \mathcal{B}_{\tau} ( \psi)}{\delta \tau}
			+ \frac{i}{8 \tau_{2}} \sum_{j = 1}^{r} \left( 2i \left( z_{j+r} \frac{\partial}{\partial z_{j}} - z_{j} \frac{\partial}{\partial z_{j+r}} \right) - \frac{t}{2} \left( \delta_{j}^{2} + \mu_{j}^{2} \right) \right) \mathcal{B}_{\tau} (\psi) \, ,
		\end{split}
	\end{equation}
	which agrees with~\eqref{eq:L2explicit}.
\end{proof}

\section{Bargmann transform on the moduli spaces}
\label{sec:Barg_below}

Throughout this section we assume $\tau$ to be fixed.
We shall discuss a version of the Bargmann transform $\underline{\mathcal{B}}_{\tau} \colon \operatorname{L}^{2}_{k} \to \mathcal{H}_{\tau,t}^{\mathbb{C}}$ on the moduli spaces and prove Thm.~\ref{thm:CH=HW}.

Recall that $W$ acts on $\mathcal{A}_{0}^{\mathbb{C}}$ by linear isometries preserving $\mathcal{A}_{0}$ and the hyper-Kähler structure, so that $P_{\tau}$ and $Q_{\tau}$ are fixed.
Fixing a $W$-invariant fundamental domain $\Dom \subset \mathcal{A}_{0}$ for $\mathcal{T}_{0}$, we obtain one for $\mathcal{A}_{0}^{\mathbb{C}}$ as $\DomC \coloneqq \Dom+P_{\tau}$, and one for the induced action on $Q_{\tau}$ as $\DomQ \coloneqq \DomC \cap Q_{\tau}$.
Finally, given $a \in \mathcal{T}_{0}$, we will denote $\bm{u}_{a}$, $\bm{q}_{a}$, $\bm{p}_{a}$, and $\bm{z}_{a}$ its coordinates in to the various frames on $\mathcal{A}_{0}^{\mathbb{C}}$.

Using~\eqref{eq:cocycle}, a section on $\mathcal{M}_{\fl}^{\mathbb{C}}$ is equivalent a $W$-invariant $\phi$ on $\mathcal{A}_{0}^{(\mathbb{C})}$ such that
\begin{equation}
	\phi(\bm{p}+\bm{p}_{a},\bm{q}+\bm{q}_{a}) = \phi(\bm{p},\bm{q}) \exp\left(\frac{-i\abs{t}}{2} \left(\bm{p}\cdot\bm{q}_{a} - \bm{q}\cdot\bm{p}_{a}\right)\right) \, ,
\end{equation}
for $a\in \mathcal{T}_{0}$.
A $P_{\tau}$-polarised one corresponds to a $W$-invariant $\psi$ on $Q_{\tau}$ with
\begin{equation}
	\psi(\bm{q}+\bm{q}_{a}) = \psi(\bm{q}') \exp\left(\frac{i\abs{t}}{2} \bm{p} \cdot (2\bm{q} + \bm{q}_{a})\right) \, .
\end{equation}
We shall often identify sections with their lifts.
Note that the kernel~\eqref{eq:Barg_kernel} satisfies
\begin{equation}
	\label{eq:Barg_ker_periodic}
	B(\bm{q}'+\bm{q}_{a}, \bm{z}) \exp \biggl( \frac{i\abs{t}}{2} (2 \bm{q}' \cdot \bm{p}_{a} + \bm{p}_{a} \cdot \bm{q}_{a}) \biggr) \!=\!
	B(\bm{q}', \bm{z}-\bm{z}_{a}) \exp \biggl(-\frac{i\abs{t}}{2}(\bm{p}\cdot\bm{q}_{a}-\bm{q}\cdot \bm{p}_{a})\biggr) \, .
\end{equation}

\begin{prop}
	\label{prop:eval_functional}
	For every fixed $\bm{z}$, the functional $T_{\bm{z}} \colon \operatorname{L}^{2,\mathbb{C}}_{t} \to \mathbb{C}$ defined by
	\begin{equation}
		T_{\bm{z}} \phi \coloneqq \left( \frac{\abs{t}}{2\pi}\right)^{\frac{r}{2}} \int_{\mathcal{A}_{0}^{\mathbb{C}}} \phi(\bm{z}') e^{-\frac{\abs{t}}{4} \left(\abs{\bm{z}}^{2} -2\bm{z}\cdot\bm{z}'+\abs{\bm{z}'}^{2}\right)} \dif \bm{z}'
	\end{equation}
	is bounded, and it restricts on $\mathcal{H}_{\tau,t}^{\mathbb{C}}$ to the evaluation at $\bm{z}$.
\end{prop}

\begin{proof}
	By uniform convergence in $\bm{z}'$ on compact sets, the sum
	\begin{equation}
		R_{\bm{z}} (\bm{z}') \coloneqq \sum_{a \in \mathcal{T}_{0}} \abs{e^{-\frac{\abs{t}}{4} \left(\abs{\bm{z}}^{2} -2\bm{z}\cdot(\bm{z}'+\bm{z}_{a})+\abs{\bm{z}'+\bm{z}_{a}}^{2}\right)}}
		= \sum_{a \in \mathcal{T}_{0}} e^{-\frac{\abs{t}}{4} \abs{\bm{z}-\bm{z}'-\bm{z}_{a}}^{2}}
	\end{equation}
	defines a $\mathcal{T}_{0}$-periodic smooth function on $\mathcal{A}_{0}^{\mathbb{C}}$, thus descending to $\mathcal{M}_{\fl}^{\mathbb{C}}$.
	Using $\mathcal{T}_{0} \subset \mathcal{A}_{0}$ and $\mathcal{A}_{0} \perp J \mathcal{A}_{0}$, the above can be expressed in $\tau$-independent coordinates as $R'_{\bm{z}}(\bm{v}') \exp\bigl(-\frac{\abs{t}}{4} \abs{\bm{u}'}_{g_{\tau}^{\mathbb{C}}}^{2}\bigr)$ for a smooth periodic $R'_{\bm{z}}$.
	Therefore $R_{\bm{z}}$ is $\operatorname{L}^{2}$ on $\mathcal{M}_{\fl}^{\mathbb{C}}$, so $\abs{T_{\bm{z}}\phi}$ is bounded by the $\operatorname{L}^{2}$-product of $\abs{\phi}$ and $R_{\bm{z}}$, showing continuity.
	
	It is well known~\cite{woodhouse_1980_geometric_quantisation} that $T_{\bm{z}} \varphi = \varphi(\bm{z})$ for $\varphi \in \widetilde{\mathcal{H}}_{\tau,t}^{\mathbb{C}}$.
	The proof uses only the Cauchy formula and Fubini-Tonelli, and holds for holomorphic $\varphi$ provided the integral converges absolutely, which we just checked to be the case.
\end{proof}

\begin{thm}
	There is a unitary linear mapping $\underline{\mathcal{B}}_{\tau} \colon \operatorname{L}^{2}_{k} \to \mathcal{H}_{\tau,t}^{\mathbb{C}}$ given by
	\begin{equation}
		\label{eq:bargbelow}
		\left( \underline{\mathcal{B}}_{\tau} (\psi) \right) (\bm{z}) \coloneqq \int_{Q_{\tau}} \psi (\bm{q}') B(\bm{q}',\bm{z}) \dif \vol \bm{q}' \, .
	\end{equation}
\end{thm}

\begin{proof}
	By a similar argument as in in the proof of Prop.~\ref{prop:eval_functional}, for fixed $\bm{z}$ the integral defines a bounded functional $S_{\bm{z}} \colon \operatorname{L}^{2}_{k} \to \mathbb{C}$.
	Fixing $\psi$ and varying $\bm{z}$ defines a holomorphic section $\varphi$ of $L_{t}^{\mathbb{C}} \to \mathcal{A}_{0}^{\mathbb{C}}$.
	Its $\mathcal{K}_{0}$-equivariance follows from that of $\psi$ by changing variables and using~\eqref{eq:Barg_ker_periodic} and the fact that $W$ acts by reflections.

	Suppose now that $\psi$ is smooth.
	Then $\abs{\psi}$ is bounded on $\mathcal{M}_{\fl}$ and therefore on $Q_{\tau}$.
	In order to compute $\norm{\varphi}^{2}$, we first consider $\varphi_{\lambda}(\bm{z}) \coloneqq \varphi(\bm{z}) \exp(-\lambda^{2}\abs{t} \abs{\bm{p}}^{2})$ for positive $\lambda$, and the inner products
	\begin{equation}
		\begin{split}
			\braket{\varphi_{\lambda}, \varphi_{\mu}} = \frac{\abs{t}^{3r}}{4^{r} \pi^{3r} \abs{W}} \int\limits_{\DomC} \int\limits_{Q_{\tau}} \int\limits_{Q_{\tau}} &
				\psi(\bm{q}') \overline{\psi(\bm{q}'')}
				e^{-\frac{\abs{t}}{2}\bigl( \abs{\bm{q}-\bm{q}'}^{2} + \abs{\bm{q}-\bm{q}''}^{2}\bigr)} \cdot \\
				& \cdot e^{-\frac{\abs{t}}{2} \bigl((\lambda^{2}+\mu^{2}) \abs{\bm{p}}^{2} - 2i\bm{p}\cdot(\bm{q}'-\bm{q}'')\bigr)} \dif \bm{q}' \dif\bm{q}'' \dif \bm{z} \, .
		\end{split}
	\end{equation}
	By absolute convergence we may apply Fubini-Tonelli, and integrating in $\bm{p}$ yields
	\begin{equation}
		\begin{split}
			\braket{\varphi_{\lambda},\varphi_{\mu}} = \frac{\abs{t}^{2r}}{2^{r} \pi^{2r} \abs{W}(\lambda^{2}+\mu^{2})^{r}} \int\limits_{\DomQ} \int\limits_{Q_{\tau}} \int\limits_{Q_{\tau}} &
			\psi(\bm{q}') \overline{\psi(\bm{q}'')}
			e^{-\frac{\abs{t}}{2}\bigl( \abs{\bm{q}-\bm{q}'}^{2} + \abs{\bm{q}-\bm{q}''}^{2}\bigr)} \cdot \\
			& \cdot e^{-\frac{\abs{t}}{2(\lambda^{2}+\mu^{2})} \abs{\bm{q}'-\bm{q}''}^{2}} \dif \bm{q}' \dif\bm{q}'' \dif \bm{q} \, .
		\end{split}
	\end{equation}
	Setting $\bm{q}' = \bm{\xi}+\alpha\bm{\eta}$ and $\bm{q}'' =\bm{\xi}-\alpha\bm{\eta}$ for $\alpha \coloneqq \sqrt{\frac{\lambda^{2} + \mu^{2}}{\lambda^{2}+\mu^{2}+2}}$ gives
	\begin{equation}
		\begin{split}
			\braket{\varphi_{\lambda}, \varphi_{\mu}} = \frac{2^{r}\abs{t}^{2r}}{\pi^{2r} \abs{W}(\lambda^{2}+\mu^{2}+2)^{r}} \int\limits_{\DomQ} \int\limits_{Q_{\tau}} \int\limits_{Q_{\tau}} &
			\psi(\bm{\xi}+\alpha\bm{\eta}) \overline{\psi(\bm{\xi}-\alpha\bm{\eta})} \cdot \\
			& \cdot e^{-\abs{t}\bigl(\abs{\bm{q}-\bm{\xi}}^{2}+\abs{\bm{\eta}}^{2}\bigr)} \dif \bm{\xi} \dif\bm{\eta} \dif \bm{q} \, .
		\end{split}
	\end{equation}
	By the continuity of $\psi$ and dominated convergence we then have
	\begin{equation}
		\begin{split}
			L \coloneqq \lim_{(\lambda,\mu) \to (0,0)} \braket{\varphi_{\lambda}, \varphi_{\mu}}
			={}& \frac{\abs{t}^{2r}}{\pi^{2r}\abs{W}} \int\limits_{\DomQ} \int\limits_{Q_{\tau}} \int\limits_{Q_{\tau}} \abs{\psi(\bm{\xi})}^{2}
			e^{-\abs{t}\bigl(\abs{\bm{q}-\bm{\xi}}^{2}+\abs{\bm{\eta}}^{2}\bigr)} \dif \bm{\xi} \dif\bm{\eta} \dif \bm{q} = \\
			={}& \frac{\abs{t}^{r}}{\pi^{r}\abs{W}} \int\limits_{\DomQ} \int\limits_{Q_{\tau}} \abs{\psi(\bm{\xi})}^{2}
			e^{-\abs{t}\abs{\bm{q}-\bm{\xi}}^{2}} \dif \bm{\xi} \dif \bm{q} \, .
		\end{split}
	\end{equation}
	Although $\bm{q}$ runs over $\DomQ$, we can use the periodicity of $\abs{\psi(\bm{\xi})}^{2}$ to obtain
	\begin{equation}
		\begin{split}
			L ={}& \frac{\abs{t}^{r}}{\pi^{r}\abs{W}} \sum_{a \in \mathcal{T}_{0}} \int\limits_{\DomQ} \int\limits_{\DomQ} \abs{\psi(\bm{\xi})}^{2}
			e^{-\abs{t}\abs{\bm{q}+\bm{q}_{a}-\bm{\xi}}^{2}} \dif \bm{\xi} \dif \bm{q} =\\
			={}& \frac{\abs{t}^{r}}{\pi^{r}\abs{W}} \int\limits_{\DomQ} \abs{\psi(\bm{\xi})}^{2} \biggl(\int_{Q_{\tau}} e^{-\abs{t}\abs{\bm{q}+\bm{q}_{a}-\bm{\xi}}^{2}} \dif\bm{q} \biggr) \dif\bm{\xi} = \frac{1}{\abs{W}} \int_{\DomQ} \abs{\psi(\bm{\xi})}^{2} \dif\bm{\xi} = \norm{\psi}^{2} \, .
		\end{split}
	\end{equation}
	
	We proved that $\braket{\varphi_{\lambda},\varphi_{\mu}}$ tends to $\norm{\psi}^{2}$ when $(\lambda,\mu) \to (0,0)$, so $\norm{\varphi_{\lambda}-\varphi_{\mu}}$ tends to $0$.
	By completeness of $\operatorname{L}^{2,\mathbb{C}}_{t}$, we can conclude that $\varphi_{\lambda}$ has a limit in that space, with norm $\norm{\psi}^{2}$.
	Since $\varphi_{\lambda}$ has pointwise limit $\varphi$, the two have to coincide.
	
	We have shown that the restriction $U$ of $\underline{\mathcal{B}}_{\tau}$ to the smooth sections is unitary.
	It remains to prove that the continuous extension of $U$ to $\operatorname{L}^{2}_{k}$ coincides with $\underline{\mathcal{B}}_{\tau}$.
	For every $\bm{z}$, however, the tautological identity $T_{\bm{z}} \circ U = S_{\bm{z}}$ holds on smooth sections, and extends by continuity to all of $\operatorname{L}^{2}_{k}$.
\end{proof}

\begin{thm}
	There is a unique bounded map $\underline{\mathcal{B}}'_{\tau} \colon \operatorname{L}^{2,\mathbb{C}}_{t} \to \operatorname{L}^{2}_{k}$ defined on an appropriate dense subspace by
	\begin{equation}
		\bigl(\underline{\mathcal{B}}'_{\tau} (\phi)\bigr)(\bm{q}) \coloneqq \int_{\mathcal{A}_{0}^{\mathbb{C}}} \phi(\bm{z}') \overline{B(\bm{q},\bm{z}')} \dif \bm{z}' \, .
	\end{equation}
	Furthermore, if $\varphi \in \mathcal{H}_{\tau,t}^{\mathbb{C}}$ then $\underline{\mathcal{B}}_{\tau} \bigl(\underline{\mathcal{B}}'_{\tau} (\varphi)\bigr) = \varphi$.
\end{thm}

Throughout the proof, the symbol $\doteq$ will mean that two quantities agree up to a constant normalisation which may depend on $t$, $r$, and $\abs{W}$, but nothing else.

\begin{proof}
	Fix a smooth compactly supported section $\phi$ on $\mathcal{M}_{\fl}^{\mathbb{C}}$, for which the integral converges absolutely.
	For each $\bm{q}$ call $\phi_{\bm{q}}$ its restriction to $P_{\tau} + \bm{q}$ and consider
	\begin{equation}
		\bigl(\mathcal{F}(\phi_{\bm{q}})\bigr)(\bm{\xi}) \coloneqq \int_{P_{\tau}} \phi(\bm{z}) e^{-\frac{i\abs{t}}{2} \bm{p} \cdot \bm{\xi}} \dif \bm{\xi}
	\end{equation}
	Using a change of variable and the quasi-periodicity of $\phi$ we can write
	\begin{equation}
		\label{eq:Fourier_periodicity}
		\begin{split}
			\bigl(\mathcal{F}(\phi_{\bm{q}+\bm{q}_{a}})\bigr)(\bm{\xi}) ={}& \int_{P_{\tau}} \phi(\bm{z}+\bm{z}_{a}) e^{-\frac{i\abs{t}}{2}(\bm{p}+\bm{p}_{a}) \cdot \bm{\xi}} \dif\bm{\xi} =\\
			={}& \int_{P_{\tau}} \phi(\bm{z}) e^{-\frac{i\abs{t}}{2} \bigl(\bm{p}\cdot(\bm{\xi}+\bm{q}_{a}) + \bm{p}_{a} \cdot(\bm{\xi}-\bm{q})\bigr)} \dif\bm{p} = \\
			={}& \bigl(\mathcal{F}(\phi_{\bm{q}})\bigr)(\bm{\xi}+\bm{q}_{a}) e^{-\frac{i\abs{t}}{2} \bm{p}_{a} \cdot(\bm{\xi}-\bm{q})} \, .
		\end{split}
	\end{equation}
	By the unitarity of the Fourier transform we can write
	\begin{equation}
		\norm{\phi}^{2}
		\doteq \int\limits_{\DomQ} \int_{P_{\tau}} \abs{\phi(\bm{z})}^{2} \dif\bm{p} \dif\bm{q}
		\doteq \int\limits_{\DomQ} \int_{Q_{\tau}} \abs{\bigl(\mathcal{F}(\phi_{\bm{q}})\bigr)(\bm{\xi})}^{2} \dif\bm{\xi} \dif\bm{q} \, .
	\end{equation}
	Change now variable to $\bm{q} - 2 \bm{\xi}$ and use~\eqref{eq:Fourier_periodicity} to obtain
	\begin{equation}
		\begin{split}
			\norm{\phi}^{2} \doteq{}& \int_{\DomQ} \int_{Q_{\tau}} \abs{\bigl(\mathcal{F}(\phi_{\bm{q}})\bigr)(\bm{q}-2\bm{\xi})}^{2} \dif\bm{\xi} \dif\bm{q} \doteq \\
			\doteq{}& \sum_{a \in \mathcal{T}_{0}} \int_{\DomQ}\int_{\DomQ} \abs{\bigl(\mathcal{F}(\phi_{\bm{q}})\bigr)(\bm{q}-2\bm{\xi}-2\bm{q}_{a})}^{2} \dif\bm{\xi} \dif\bm{q} =\\
			={}& \sum_{a \in \mathcal{T}_{0}} \int_{\DomQ}\int_{\DomQ} \abs{\bigl(\mathcal{F}(\phi_{\bm{q}-\bm{q}_{a}})\bigr)((\bm{q}-\bm{q}_{a})-2\bm{\xi})}^{2} \dif\bm{\xi} \dif\bm{q} = \\
			={}& \int_{Q_{\tau}} \int_{\DomQ} \abs{\bigl(\mathcal{F}(\phi_{\bm{q}})\bigr)(\bm{q}-2\bm{\xi})}^{2} \dif\bm{\xi} \dif\bm{q} \, .
		\end{split}
	\end{equation}
	
	We are ready to study the $\operatorname{L}^{2}$-norm of $\underline{\mathcal{B}}'_{\tau}(\phi)$.
	Using Cauchy-Schwarz we obtain
	\begin{equation}
		\begin{split}
			\abs{\bigl(\underline{\mathcal{B}}'_{\tau} (\phi)\bigr)(\bm{q})}^{2} ={}&
			\abs{\int_{\mathcal{A}_{0}^{\mathbb{C}}} \phi(\bm{z}') e^{-\frac{\abs{t}}{2} \abs{\bm{q}-\bm{q}'}^{2}} e^{\frac{i\abs{t}}{2} \bm{p}'\cdot(2\bm{q}-\bm{q}')} \dif\bm{z}'}^{2} \doteq\\
			\doteq{}& \abs{\int_{Q_{\tau}} \bigl(\mathcal{F}(\phi_{\bm{q}'})\bigr)(\bm{q}'-2\bm{q}) e^{-\frac{\abs{t}}{2} \abs{\bm{q}-\bm{q}'}^{2}} \dif\bm{q}'}^{2} \leq\\
			\leq{}& C \int_{Q_{\tau}} \abs{\bigl(\mathcal{F}(\phi_{\bm{q}'})\bigr)(\bm{q}'-2\bm{q}) }^{2}\dif\bm{q}'
		\end{split}
	\end{equation}
	for some positive constant $C$.
	But then
	\begin{equation}
		\norm{\underline{\mathcal{B}}'_{\tau}(\phi)}^{2}_{\operatorname{L}^{2}_{k}}
		\leq \frac{C}{\abs{W}} \int_{\DomQ} \int_{Q_{\tau}} \abs{\bigl(\mathcal{F}(\phi_{\bm{q}'})\bigr)(\bm{q}'-2\bm{q}) }^{2}\dif\bm{q}' \dif\bm{q}
		\doteq \norm{\phi}^{2}_{\operatorname{L}^{2,\mathbb{C}}_{t}} \, .
	\end{equation}
	Therefore $\underline{\mathcal{B}}'$ is bounded on the dense space of smooth compactly supported sections.
	On this dense subspace, moreover, Fubini-Tonelli applies for every $\bm{z}$, giving
	\begin{equation}
		\begin{split}
			\Bigl(\underline{\mathcal{B}}_{\tau}\bigl(\underline{\mathcal{B}}'_{\tau}(\phi)\bigr)\Bigr)(\bm{z}) ={}&
			\int_{\mathcal{A}_{0}} \int_{\mathcal{A}_{0}^{\mathbb{C}}} \phi(\bm{z}') \overline{B(\bm{q}'',\bm{z}')} B(\bm{q}'', \bm{z}) \dif\bm{z}' \dif\bm{q}'' =\\
			={}& \int_{\mathcal{A}_{0}^{\mathbb{C}}} \phi(\bm{z}') \int_{\mathcal{A}_{0}} \overline{B(\bm{q}'',\bm{z}')} B(\bm{q}'', \bm{z}) \dif\bm{q}'' \dif\bm{z}' = T_{\bm{z}} \phi \, .
		\end{split}
	\end{equation}
	It follows from continuity that $S_{\bm{z}} \circ \underline{\mathcal{B}}'_{\tau} = T_{\bm{z}}$ on $\operatorname{L}^{2,\mathbb{C}}_{t}$, and in particular on $\mathcal{H}_{\tau,t}^{\mathbb{C}}$.
\end{proof}

\section{Identifications of the connections on the moduli spaces}
\label{sec:duals}

In this section we use the previous results to identify the Hitchin--Witten and complexified Hitchin connections as intrinsically defined on the moduli spaces.
The arguments that used the $\operatorname{L}^2$-property can be adapted building on~\cite[\S~3.5]{andersen_1992_jones_witten_theory}.
We recall the convention that polarised sections, without further specifications, will refer to $P_{\tau}$---Kähler-polarised objects will be called holomorphic.

Let $\mathcal{S} = \mathcal{S}_{k}$ be the space of smooth sections of ${\mathcal{L}}_{k} \to \mathcal{A}_{0}$ whose point-wise norm squared function is Schwartz-class.
Analogously, let $\mathcal{S}_{\mathbb{C}} = \mathcal{S}_{\tau,t,\mathbb{C}}$ be the space of Schwartz-class holomorphic sections of ${\mathcal{L}}^{\mathbb{C}}_{t} \to \mathcal{A}_{0}^{\mathbb{C}}$.
These spaces embed densely inside the quantum spaces $\widetilde{\operatorname{L}}^{2}_{k}$ and $\widetilde{\mathcal{H}}^{\mathbb{C}}_{\tau,t}$ of \S\S~\ref{sec:real_quantisation} and~\ref{sec:kaehler_quantisation}, respectively, for $\tau \in \mathcal{T}$.

As a consequence of its fundamental properties, the Bargmann transform restricts to a morphism of Fr\'{e}chet spaces $\mathcal{B}_{\tau} \colon \mathcal{S} \to \mathcal{S}_{\mathbb{C}}$, defining a \emph{transpose} map $\prescript{t}{}{\mathcal{B}}_{\tau} \colon \mathcal{S}'_{\mathbb{C}} \to \mathcal{S}'$ between the topological duals.
It is also easy to check that the lifts of elements of $\operatorname{L}_{k}^{2}$ and $\mathcal{H}^{\mathbb{C}}_{\tau,t}$ have finite $\operatorname{L}^{2}$-pairing with elements of $\mathcal{S}$ and $\mathcal{S}_{\mathbb{C}}$, respectively, resulting in embeddings $\iota$ and $\iota^{\mathbb{C}}$ as in the diagram in Fig.~\ref{fig:diagram}.

\begin{figure}[h]
	\begin{tikzpicture}
		\node (H) at (0,0) {$\operatorname{L}^{2}_{k}$};
		\node (HC) at (3,0) {$\mathcal{H}_{\tau,t}^{\mathbb{C}}$};
		\node (S) at (0,-2) {$\mathcal{S}_{k}'$};
		\node (SC) at (3,-2) {$\mathcal{S}'_{\tau,t,\mathbb{C}}$};
		
		\draw[->] (H) -- node[above]{$\underline{\mathcal{B}}_{\tau}$} (HC);
		\draw[->] (H) -- node[left]{$\iota$}(S);
		\draw[->] (SC) -- node[below]{$\prescript{t}{}{\mathcal{B}}_{\tau}$} (S);
		\draw[->] (HC) -- node[right]{$\iota^{\mathbb{C}}$}(SC);
	\end{tikzpicture}
	\caption{Comparison between polarised sections and tempered distributions}
	\label{fig:diagram}
\end{figure}

\begin{lem}
	The diagram of Fig.~\ref{fig:diagram} is commutative.
\end{lem}

\begin{proof}
    We must show that $\bigl( \underline{\mathcal{B}}_{\tau} (\psi_1) \bigm| \mathcal{B}_{\tau} (\psi_2) \bigr) = ( \psi_1 | \psi_2 )$, for $\psi_1 \in \operatorname{L}^{2}_{k}$ and $\psi_2 \in \mathcal{S}$.
    
	For every $a \in \mathcal{T}_{0}$ let $\chi_{a}$ denote the indicator function of $\mathcal{D}+a$, so that $\psi \chi_{a}$ is $\operatorname{L}^2$ on $\mathcal{A}_{0}$.
	By dominated convergence and unitarity of $\mathcal{B}_{\tau}$, we then have
	\begin{equation}
		( \psi_1 | \psi_2 ) = \sum_{a \in \mathcal{T}_{0}} ( \psi_1 \chi_{a} | \psi_2 ) = \sum_{a \in \mathcal{T}_{0}} \bigl( \mathcal{B}_{\tau} (\psi_1 \chi_{a}) \bigm| \mathcal{B}_{\tau} (\psi_2) \bigr) \, .
	\end{equation}
	Dominated convergence also yields $\sum_{a} \mathcal{B}_{\tau} \bigl( \psi_1 \chi_{a} \bigr) = \underline{\mathcal{B}}_{\tau} (\psi_{1})$, point-wise on $\mathcal{A}_{0}^{\mathbb{C}}$, and all finite partial sums are uniformly bounded in absolute by a constant, so
	\begin{equation}
		\sum_{a \in \mathcal{T}_{0}} \bigl( \mathcal{B}_{\tau} (\psi_{1} \chi_{a}) \bigm| \mathcal{B}_{\tau} (\psi_{2}) \bigr) = \Biggl( \sum_{a \in \mathcal{T}_{0}} \mathcal{B}_{\tau} (\psi_{1} \chi_{a}) \Biggm| \mathcal{B}_{\tau} (\psi_{2}) \Biggr)
		= \bigl( \underline{\mathcal{B}}_{\tau} (\psi_{1}) \bigm| \mathcal{B}_{\tau} (\psi_{2}) \bigr) \, .
	\end{equation}
\end{proof}

Now we consider the dual versions of the Hitchin--Witten and complexified Hitchin connections, as follows.
Suppose $T$ is a $\mathcal{T}$-family of elements of $\mathcal{S}'$, such that for every test section $\psi \in \mathcal{S}$ the pairing $(T | \psi)$ is smooth over $\mathcal{T}$.
Then set
\begin{equation}
\label{eq:dual_HW_connection}
	\bigl( \check{\nabla}^{\HW}_{V} T \bigm| \psi \bigr) \coloneqq V \left[ ( T | \psi ) \right] - \bigl( T \bigm| \nabla^{\HW}_{V} \psi \bigr) \, ,
\end{equation}
where $V$ is a vector field on $\mathcal{T}$.
This defines the \emph{dual} Hitchin--Witten connection.

The \emph{dual} complexified Hitchin connection $\check{\nabla}^{\mathbb{C}}$ is defined analogously, with the caveat that the test section $\psi \in \mathcal{S}_{\mathbb{C}}$ needs to be extended to a $\tau$-dependent family in order for the right-hand side to make sense.

\begin{lem}
	The dual Bargmann transform intertwines the dual Hitchin--Witten and the complexified Hitchin connections.
\end{lem}

\begin{proof}
	Let $V$ be a vector field on $\mathcal{T}$, $T$ a smooth $\mathcal{T}$-family of elements of $\mathcal{S}'_{\mathbb{C}}$, $\psi \in \mathcal{S}$ a test section.
	By~\eqref{eq:dual_HW_connection} and the defining property of the transpose map, one has
	\begin{equation}
		\Bigl( \prescript{t}{}{\mathcal{B}}_{\tau} \bigl( \check{\nabla}^{\mathbb{C}}_{V} T \bigr) \Bigm| \psi \Bigr) 
        = V \left[ \bigl( T \bigm| \mathcal{B}_{\tau} (\psi) \bigr) \right] - \bigl( T \bigm| \nabla^{\mathbb{C}}_{V} \mathcal{B}_{\tau} (\psi) \bigr) \, .
	\end{equation}

	Since $\psi$ is a $\tau$-independent function the sections $\frac{\delta \psi}{\delta \tau}$ and $\frac{\delta \psi}{\delta \overline{\tau}}$ can be expressed in terms of derivatives along $\mathcal{A}_{0}^{\mathbb{C}}$ and coordinate multiplication.
	Therefore, $\frac{\delta \psi}{\delta \tau}$ and $\frac{\delta \psi}{\delta \overline{\tau}}$ are still Schwartz-class, and a standard argument using dominated convergence implies that the hypotheses of Thm.~\ref{thm:l^2=HWfull} hold for $\psi$.
	It follows that
	\begin{equation}
			\Bigl( \prescript{t}{}{\mathcal{B}}_{\tau} \bigl( \check{\nabla}^{\mathbb{C}}_{V} T \bigr) \Bigm| \psi \Bigr)
			= \Bigl( \check{\nabla}^{\HW}_{V} \prescript{t}{}{\mathcal{B}}_{\tau} (T)  \Bigm| \psi \Bigr) \, .
	\end{equation}
\end{proof}

Thus far we have proved that the dual Bargmann transform intertwines $\check{\nabla}^{\HW}$ and $\check{\nabla}^{\mathbb{C}}$.
In order to conclude the proof of Thm.~\ref{thm:CH=HW}, all that is left to do is to relate these two connections, obtained working by duality on the cover $\mathcal{A}_{0}^{(\mathbb{C})}$, with the connections defined intrinsically on the moduli spaces.

\begin{prop}
	The embeddings $\iota$ and $\iota^{\mathbb{C}}$ intertwine the Hitchin-Witten and complexified Hitchin connections $\nabla^{\HW}$ and $\nabla^{\mathbb{C}}$ with the duals of their lifted versions.
\end{prop}

\begin{proof}
	By~\eqref{eq:dual_HW_connection}, if $\psi$ is a smooth $\mathcal{T}$-family of $\mathcal{K}_{0}$-equivariant sections on $\mathcal{A}_{0}$ and $\psi_{0} \in \mathcal{S}'$ is a fixed test function we have that
	\begin{equation}
		\Bigl( \check{\nabla}^{\HW}_{V} (\iota \psi) \Bigm| \psi_{0} \Bigr)
		= V \left[ \int_{\mathcal{A}_{0}} \psi \cdot \overline{\psi}_{0} \dif \vol \right] - \int_{\mathcal{A}_{0}} \psi \cdot \overline{ \nabla^{\HW}_{V} \psi_{0}} \dif \vol \, ,
	\end{equation}
	for every \emph{real} tangent vector $V$ on $\mathcal{T}$.
	Again by compactness of $\mathcal{M}_{\fl}$, $\psi$ is bounded, uniformly in $\tau$ up to restricting to appropriate open subsets of $\mathcal{T}$, and the same applies to its derivatives along the direction of $V$.
	By dominated convergence then 
	\begin{equation}
		V \left[ \int_{\mathcal{A}_{0}} \psi \cdot \overline{\psi}_{0} \dif \vol \right] = \int_{\mathcal{A}_{0}} V \bigl[ \psi \cdot \overline{\psi}_{0} \bigr] \dif \vol = \bigl( V[\psi] \bigm| \psi_{0} \bigr) \, .
	\end{equation}
	On the other hand, the action of the Hitchin-Witten connection on $\psi$ reduces to that of its potential, and we obtain
	\begin{equation}
		\int_{\mathcal{A}_{0}} \psi \cdot \overline{ \nabla^{\HW}_{V} \psi_{0}} \dif \vol = \int_{\mathcal{A}_{0}} \left( u^{\HW} (V) \psi \right) \cdot \overline{\psi}_{0} \dif \vol = \bigl( \iota \bigl( u^{\HW}(V) \psi \bigr) \bigm| \psi_{0} \bigr) \, ,
	\end{equation}
	by integration by parts---using the fast decay of $\psi_{0}$.
	Overall
	\begin{equation}
		\Bigl( \check{\nabla}^{\HW}_{V} (\iota \psi) \Bigm| \psi_{0} \Bigr) = \Bigm( V[\psi] - \frac{1}{2} u^{\HW} (V) \psi \Bigm| \psi_{0} \Bigr) = \Bigl( \iota \Bigl(\nabla^{\HW}_{V} \psi \Bigr) \Bigm| \psi_{0} \Bigr) \, .
	\end{equation}

	For the K\"{a}hler-polarised case, a test section $\varphi \in \mathcal{S}'_{\mathbb{C}}$ may not be fixed independently of $\tau$, but requiring that $\frac{\delta \varphi}{\delta \tau} = \frac{\delta \varphi}{\delta \tau} = 0$ the same arguments apply.
\end{proof}

Putting together the statements of this section we have proved Thm.~\ref{thm:CH=HW}.

\bibliographystyle{amsplain}

\providecommand{\bysame}{\leavevmode\hbox to3em{\hrulefill}\thinspace}
\providecommand{\MR}{\relax\ifhmode\unskip\space\fi MR }
\providecommand{\MRhref}[2]{%
  \href{http://www.ams.org/mathscinet-getitem?mr=#1}{#2}
}
\providecommand{\href}[2]{#2}

\end{document}